\documentclass[leqno,11pt]{article}
\usepackage{latexsym,amsmath,amsfonts,mathrsfs,amssymb,amsthm}
\usepackage{yhmath}
\usepackage[usenames]{color}
\usepackage{graphicx}
\usepackage{setspace}
\usepackage{mathtools,bm,booktabs,array,tabularx}
\usepackage{xcolor}
\usepackage{tikz}
\usetikzlibrary{arrows.meta,calc,positioning,matrix,fit}
\usepackage{float}
\usepackage{caption,subcaption,enumitem}
\usepackage[hidelinks,unicode]{hyperref}

\makeatletter
\renewcommand\normalsize{%
   \@setfontsize\normalsize\@xpt\@xiipt
   \abovedisplayskip 3\p@ \@plus2\p@ \@minus5\p@
   \abovedisplayshortskip \z@ \@plus3\p@
   \belowdisplayshortskip 6\p@ \@plus3\p@ \@minus3\p@
   \belowdisplayskip \abovedisplayskip
   \let\@listi\@listI}
\makeatother
\def\r{\mathbb R}
\def\R{\mathbb R}

\def\C{\mathbb C}
\def\Z{\mathbb Z}
\def\s{\sharp}
\def\a{\alpha}

\def\d{\delta}
\def\Y{\mathbb Y}
\def\T{\mathbb T}

\def\be{\begin{equation}}
\def\ee{\end{equation}}
\def\bs{\backslash}
\def\qed{\hfill$\Box$\medskip}
\def\pa{\partial}
\def\nd{\noindent Proof. }
\def\F{{\cal F} }
\def\H{{\mathscr H}}

\numberwithin{equation}{section}
\newtheorem{lem}[equation]{Lemma}
\newtheorem{pro}[equation]{Proposition}
\newtheorem{defn}[equation]{Definition}
\newtheorem{thm}[equation]{Theorem}
\newtheorem{cor}[equation]{Corollary}
\newtheorem{rem}[equation]{Remark}

\allowdisplaybreaks[2]

\newtheorem{mainthm}{Theorem}

\def\FF{{\mathbb Z_2}}
\def\I{\mathbf I}
\def\mass{\mathbf M}
\def\car{\operatorname{car}}
\def\spt{\operatorname{spt}}

\def\Span{\operatorname{span}}
\def\conv{\operatorname{conv}}

\def\op{\mathrm{op}}
\def\Gx{G^{\times}}
\def\eps{\varepsilon}

\def\chain#1{[\![#1]\!]_2}
\def\pair#1#2{\langle #1,#2\rangle}
\def\tr{\operatorname{tr}}

\def\id{\operatorname{id}}
\def\Dcal{\mathcal D}

\definecolor{inkblue}{RGB}{30,75,115}
\definecolor{pale}{RGB}{226,237,246}
\hypersetup{pdftitle={Almgren minimality of products of simplex-skeleton cones},pdfauthor={}}
\begin{document}
\bigskip

\begin{center}
{\Large\textbf{On the Almgren minimality of products of simplex-skeleton cones}}

\medskip
{\large Songsong Huang, Xiangyu Liang, Zaoyi Zhang}
%
%
%
%
%
%
%
\end{center}

\bigskip
\centerline{\large\textbf{Abstract.}}
\medskip

In this paper, we prove that for every $m,n\geq 2$, the orthogonal product $S_m\times S_n\subset\R^m\times\R^n$ is an Almgren minimal cone of dimension $m+n-2$, where $S_k\subset\R^k$ is the cone over the $(k-2)$-skeleton of a regular $k$-simplex. Lawlor and Morgan \cite{LM94} proved the minimality of each  $S_k$ by paired calibrations, while Guy David \cite{DJT, GD13} later raised the general question whether products of Almgren minimal sets remain minimal. Our theorem gives a affirmative answer to this product question within the entire Lawlor-Morgan regular-simplex family. It extends the isolated example \(\Y\times \Y\), the product of two 1-dimensional $\Y$ sets, proved in 2014 by the second author, simultaneously to all pairwise products \(S_m\times S_n\), and therefore produces a large new family of possible codimension-two singularities for Plateau's problem.

The key idea of the proof is to find a new way to combine homology with paired calibrations in higher codimension. Instead of using only the currents associated with the facets, we consider all their nonzero $\FF$ combinations. The resulting complete family has enough algebraic symmetry to control uniformly the overlaps created by noninjective deformations.

\medskip


\medskip


\setcounter{section}{-1}
\section{Introduction}

Plateau's problem, in its simplest form, asks for a surface of least area spanning a given closed curve. Its physical models are soap films. The existence, regularity and local structure of these minimizing objects are one of the central interests in geometric measure theory. 

A soap film need not be a smooth surface: several sheets may meet along a curve or at a point. Thus a formulation intended to describe these objects should allow singularities. 

Several major frameworks have been developed for this problem, such as the theory of integral currents introduced by Federer and Fleming \cite{FeFl,Fe}, the varifold approach developed by Almgren and Allard \cite{Al66,All72}, and the theory of Almgren minimal sets and Reifenberg minimal sets from a set-theoretic point of view \cite{Al76,Rei60}.

In this paper we work with minimal sets in the sense of Almgren \cite{Al76}. Roughly speaking, a relatively closed set $E\subset U\subset\R^N$ is $d$-dimensional Almgren minimal in the open set $U$ if its $d$-dimensional Hausdorff measure cannot be decreased by a Lipschitz deformation compactly supported in $U$. More precisely, we ask that
\[
\H^d(E\bs f(E))\leq\H^d(f(E)\bs E),
\]
for every admissible deformation $f$; see Definition~\ref{def:deformation} for more detail. 

A central advantage of Almgren's formulation is that it treats the competitors as sets, without imposing orientation or algebraic multiplicities. This makes it particularly well suited to modelling soap films. Indeed, the interior singularity types of reduced two-dimensional minimal sets in $\R^3$ coincide exactly with those observed in soap film experiments \cite{Ta}. On the other hand, this lack of algebraic structure also makes minimality harder to prove. Noninjective deformations may fold different parts, while their common image is counted only once. Thus estimates coming from currents, calibrations, or parametrized area, which naturally count multiplicities, do not work directly.

For the regularity, the work of Almgren \cite{Al76} and of David and Semmes \cite{DS00} give rectifiability and local Ahlfors regularity for reduced minimal sets. Thus tangent planes exist almost everywhere, and minimal sets are locally smooth minimal submanifolds around such points. 

Therefore our main interest is in the singular points, where no tangent plane exists. A standard way to study such points is to look at blow-up limits. These limits are Almgren minimal cones \cite{DJT}, so the classification of minimal cones becomes a natural first step in the study of singularities.

For two-dimensional sets in $\R^3$, the reduced minimal cones are, up to isometries, a plane, a $\Y$ set (the union of three half-planes meeting along a line at $120^\circ$), and the tetrahedral cone $\T$ (the cone over the 1-skeleton of a regular simplex centered at the origin), see the figure below. Taylor's theorem gives a local $C^{1,\alpha}$ description by these cones \cite{Ta}.

\centerline{\includegraphics[width=0.12\textwidth]{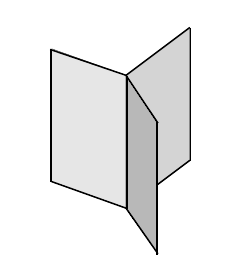} \hskip 2cm\includegraphics[width=0.15\textwidth]{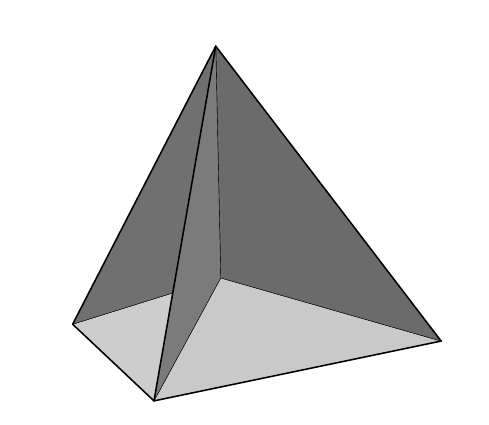}}
\nopagebreak[4]
\hskip 4.3cm a $\Y$ set\hskip 3.9cm  a $\T$ set\\
\nopagebreak[4]

In higher ambient dimension, David obtained local bi-H\"older parametrizations by two-dimensional minimal cones, and $C^{1,\alpha}$ regularity when a tangent cone satisfies the full-length condition \cite{DJT,DEpi}. Thus the cones also determine the available local regularity.

In $\R^4$ (and higher dimensional spaces), the classification is still far from clear. Besides the cones contained in $\R^3$, known examples include unions of two almost orthogonal planes \cite{2p} and $\Y\times\Y$ \cite{YXY}. Here $\Y\subset\R^2$ denotes the union of three rays meeting at $120^\circ$. The plane unions form continuous families of nonisometric cones, while $\Y\times\Y$ gives a different singular configuration. A complete classification remains largely open.

Products provide a natural way to construct new minimal cones. Guy David explicitly raised the general question whether the orthogonal product of two Almgren minimal sets is again minimal \cite[Remark~8.22]{DJT}. The underlying difficulty is that a deformation of $E_1\times E_2$ need not split into deformations of each factor, and different product pieces may overlap, so the separate minimality inequalities do not directly imply minimality of the product. 

This question on product is particularly natural for the regular-simplex cones of Lawlor and Morgan: they proved that, for every $k\ge 2$, the cone $S_k$ over the $(k-2)$-skeleton of a regular $k$-simplex is minimal by paired calibration \cite[Section~1.1]{LM94}. In fact, David further singled out the special branching product $S_2\times S_2=\Y\times \Y$ as a concrete conjectural example (\cite[Section~2]{DEpi} and \cite[Section~4.1]{GD13}); its Almgren minimality was proved by the second author in 2014 \cite{YXY} by combining paired calibrations with a decomposition governed by mod-two homology. Subsequent product theorems cover different classes, for instance a paired-calibrated factor times a codimension-one calibrated manifold \cite{product1}, or a codimension-one set (with singularity) admitting a suitable coflat calibration \cite{product2}. These yield products involving singular area-minimizing hypercones, but do not cover any case for the product of two regular-simplex cones. In fact, before the present work, no general theorem covered products of two branching members of the regular-simplex family beyond the special case $\Y\times \Y$. 

The main difficulty comes from overlaps. The classical paired calibration argument \cite{Br91,LM94}, which was used to prove the minimality for all simplex-skeleton cones,  is a codimension-one argument: the complementary regions are separated, and on each interface one only has to compare the two calibration fields corresponding to the adjacent regions. For a product of two branching sets, we are in codimension 2, and this separation is lost. After a noninjective deformation, several different product faces may be mapped onto the same set, and the sum effects of calibration fields can no longer be estimated directly by the classical paired-calibration argument.

Now let us state our main results. We first state the tetrahedral case, which already exhibits the main new difficulty:

\begin{mainthm}[Minimality of the orthogonal product of two $\T$]\label{main:min}
Let $\T,\T'\subset \R^3$ be regular tetrahedral cones centered at the origin. Then their orthogonal product
\be
E=\T\times\T'\subset \R^3\times \R^3=\R^6
\ee
is a four-dimensional Almgren minimal cone.
\end{mainthm}

In fact, similar idea can be extended to the whole regular-simplex family:

\begin{mainthm}[Minimality of the orthogonal product of two regular simplex-skeleton cones]\label{main:simplex}
For each integer $k\geq2$, let $S_k\subset\R^k$ be the cone over the $(k-2)$-skeleton of a regular $k$-simplex centered at the origin. For every $m,n\geq2$, the orthogonal product
\be
S_m\times S_n\subset\R^m\times\R^n
\ee
is an Almgren minimal cone of dimension $m+n-2$.
\end{mainthm}

Theorem \ref{main:simplex} solves, in a single argument, all pairwise products in the Lawlor-Morgan regular-simplex family. The result for $\Y\times\Y=S_2\times S_2$ established in 2014 \cite{YXY} is thus extended at once to every pair $S_m\times S_n$, $m,n\ge 2$, including $\T\times \Y$ and $\T\times \T$. We therefore obtain a complete affirmative answer to David's product question within this
family. In particular, it produces a systematic family of new codimension-two singularity models.

The key idea of the proof is to enlarge the family of currents used in the paired calibration argument. Instead of keeping only the basic currents associated with the facets, we consider all their nonzero mod-two combinations. For a $k$-simplex this gives $2^k-1$ test currents.

At first glance, most of these additional currents have no direct geometric interpretation as individual faces of the simplex. Their role becomes visible only at the algebraic level: after a noninjective deformation, different faces may overlap, and the multiplicities of the basic currents may add modulo two. The original facet family does not give a uniform control of all such overlap patterns. The complete family does: all nonzero $\FF$ linear
functionals have the same number of elements that take the value one. Combined with the symmetry of the regular simplex, this gives the quadratic symmetry identity and the pointwise estimate needed in the proof for arbitrary overlaps.

Extending this method to more factors or to general paired-calibrated sets is a natural further question; the additional-factor case is briefly discussed in Subsection~\ref{subsec:simplex-scope}.

The rest of the paper is organized as follows. In Section~\ref{sec:tetra} we recall the necessary definitions, the classical paired calibration for the tetrahedral cone $\T$, and explain the overlap obstruction for products. In Section~\ref{sec:full-cycle} we introduce the complete $\FF$ cycle family for $\T$, the associated projection and area vectors, the pointwise estimate,and use then to give a new proof of the minimality of $\T$. Section~\ref{sec:estimate} is devoted to the product structure: we introduce the local multiplicity matrix and prove the sharp pointwise estimate. In Section~\ref{sec:minimality} we prove the Almgren minimality of the tegrahedral product (Theorem~\ref{main:min}). Finally in Section~\ref{sec:simplex-products} we extend the argument to arbitrary regular simplex-skeleton cones(Theorem~\ref{main:simplex}).

\noindent\textbf{Acknowledgement:} The second author would like to thank Giovanni Alberti, Guy David, Thierry De Pauw for useful discussion. This work is supported by the General Program of the National Natural Science Foundation of China (Grant No.~12671108 and Grant No.~12271018), the Foundation for Innovative Research Groups of the Natural Science Foundation of China (No. 12121001), and Beijing Natural Science Foundation Youth Science Fund Key Project (Grant No.~JQ26002). 

\noindent\textbf{Statement on the use of AI:} The authors used OpenAI's GPT-6 Astra in developping this work. The two key estimates---the four-point support property and the initial
pointwise estimate for $\T\times\T$---were formulated by
GPT-6 Astra in a more algebraic and substantially shorter form, using the
symmetry identity (Lemma~\ref{lem:frame}) and the operator-norm estimate
(Lemma~\ref{lem:signnorm}), in place of the authors' original direct
calculations. This formulation also led naturally to the higher-dimensional versions, in particular Lemmas~\ref{lem:sp-frame} and~\ref{lem:sp-incidence} and Theorem~\ref{thm:sp-pointwise}. The geometric argument was done by the authors. The authors take full responsibility for the corectness of this work.

\section{Preliminaries}\label{sec:tetra}

In this section, we first introduce necessary notations and definitions in Subsections \ref{subsec:notation}--\ref{subsec:currents}. We then use the set $\T$ to illustrate the paired calibrations argument in Subsection \ref{subsec:classical-paired}, and explain briefly why the product of paired calibrations cannot be used directly to prove the minimality for the product of two $\T$'s.

\subsection{Notation and geometric conventions}\label{subsec:notation}

-For points $a,b\in\R^N$, write $[a,b]$ for the closed segment, $\R_+a=\{ta:t\geq0\}$ for the ray generated by $a$;

-$B(x,r)$ and $\overline B(x,r)$ denote the open and closed Euclidean balls respectively;

-The symbols $\conv$, $\Span$, denote convex hull and linear span respectively;

-The notation $A\Subset U$ means that $\overline A$ is compact and contained in the open set $U$.

-For an integer $d\le N$, and a $d$-rectifiable set $M$ of locally finite $\H^d$ measure, $T_xM$ denotes its approximate tangent plane at the point $x$, defined at almost every point. 

-A relatively closed set $M\subset U$ is $d$-reduced if
$$
 M=\{x\in M:\H^d(M\cap B(x,r))>0
       \text{ whenever }r>0\text{ and }B(x,r)\Subset U\}.
$$

-The notation $\H^d\llcorner M$ denotes the restriction of $\H^d$ to $M$.

-The space $\Lambda_j(\R^N)$ has the Euclidean inner product induced by that of $\R^N$. A simple $j$-vector is an exterior product of $j$ vectors; it is unit if its Euclidean norm is one. It represents an oriented $j$-plane, and changing its sign reverses orientation. We write $G(N,j)$ for the unoriented Grassmannian and $*$ for the Hodge star after an ambient orientation has been fixed. 

-For a constant $j$-form $\omega$, its comass is 
$$
 \|\omega\|_*=
 \sup\{|\omega(\xi)|:\xi\text{ is a unit simple }j\text{-vector}\}
$$

-The notation $v^\flat$ denotes the covector associated with a vector $v$. 

-$J_j(p|_L)$ denotes the $j$-dimensional Jacobian of a linear map $p$ restricted to a $j$-plane $L$; the same notation is used for the approximate tangential differential of a Lipschitz map. 

-For a real matrix, $\|\cdot\|_{\op}$ is its Euclidean operator norm. $\FF$ ($\FF=\Z/2\Z$) matrix operations will always be explicitly distinguished from real matrix operations.

-Homology groups and cycle spaces will always have coefficients in $\FF$ unless otherwise stated.  $\FF$ basis vectors are denoted by $e_i$ and Euclidean coordinate vectors by $\mathbf e_i$ when both occur in the same discussion. 

\subsection{Minimal sets and classes of competitors}\label{subsec:minimal-prelim}
\begin{defn}[Minimality with respect to a competitor class]\label{def:relative-minimality}
Let $0<d<n$, let $U\subset\R^n$ be open, and let $M\subset U$ be relatively closed. Let $\F$ be a class of relatively closed subsets of $U$. We say that $M$ is $d$-dimensional minimal in $U$ with respect to $\F$ if $\H^d(M\cap B)<\infty$ for every compact ball $B\subset U$ and
\be
 \H^d(M\bs F)\leq\H^d(F\bs M)
 \qquad\text{for every }F\in\F.
\ee
The sets in $\F$ are called competitors.
\end{defn}
If $\F_1\subset\F_2$, minimality with respect to $\F_2$ implies minimality with respect to $\F_1$. This fact will be used in the proof of Almgren minimality below: we first minimize in a spanning competitor class $\F_2$ and then verify that the class $\F_1$ of Almgren deformation is included in $\F_2$. See Sections \ref{sec:minimality} and \ref{subsec:simplex-boundary}-\ref{subsec:simplex-almgren}. 

\begin{defn}[Almgren deformations, competitors, and minimizers]\label{def:deformation}
Let $U\subset\R^n$ be open. 

$1^\circ$ A family $\varphi_t:U\to U$, $0\leq t\leq1$, is an Almgren deformation in $U$ if $(t,x)\mapsto\varphi_t(x)$ is continuous, $\varphi_0=\id$, $\varphi_1$ is Lipschitz, and
\be\label{eq:track}
\widehat W=\bigcup_{0\leq t\leq1}\left(W_t\cup\varphi_t(W_t)\right)\Subset U\mbox{, where }
 W_t=\{x\in U:\varphi_t(x)\ne x\}.
\ee

$2^\circ$ For every relatively closed set $M\subset U$, an Almgren competitor of $M$ is a set $F=\varphi_1(M)$ for some Almgren deformation $\varphi_t, 0\leq t\leq1$ in $U$. 

$3^\circ$ A relatively closed set $M\subset U$ with locally finite $\H^d$-measure is called $d$-dimensional Almgren minimal in $U$ if it is $d$-dimensional minimal in $U$ with respect to the class of all its Almgren competitors.
\end{defn}

\begin{lem}\label{lem:ball-testing}
Suppose $U$ is either $\R^N$ or an open ball. A relatively closed set $M\subset U$ of locally finite $\H^d$-measure is Almgren minimal if and only if, for every closed ball $B\Subset U$ and every Lipschitz map $f:U\to U$ satisfying
\be
 f(B)\subset B,\qquad f=\id\quad\text{on }U\bs B,
\ee
one has
\be\label{eq:ball-comparison}
 \H^d(M\cap B)\leq\H^d(f(M)\cap B)
       =\H^d(f(M\cap B)).
\ee
\end{lem}
\nd

$\Rightarrow$:
For such a $f$, the straight line homotopy $\varphi_t(x)=(1-t)x+tf(x)$ is an Almgren deformation in $U$ with $f=\varphi_1$. Hence the Almgren minimality of $M$ implies immediately \eqref{eq:ball-comparison}.

$\Leftarrow$: for any Almgren deformation $\varphi_t,0\le t\le 1$, just take $\varphi_1=f$. The compact moving track condition \eqref{eq:track} guarantees the existence of a ball $B$ for $f$, and \eqref{eq:ball-comparison} implies the Almgren minimality.
\qed
%

\subsection{Currents modulo two and elementary facts}\label{subsec:currents}
Unless otherwise specified, the coefficient group is always $\FF$, with its nonzero element assigned norm one. We write $\I_k(\R^n;\FF)$ for the space of $k$-dimensional integral currents modulo two in the sense of Fleming
\cite{Fl66}. Unless otherwise stated, all the currents considered below have compact support.

A $k$-rectifiable current modulo two has a representation
\be
A=\chain{M,\theta},\qquad \theta:M\to\{0,1\},
\ee
where $M$ is $k$-rectifiable. Its essential carrier is
$\car A=\{x\in M:\theta(x)=1\}$, well defined up to $\H^k$-null sets, and
$\mass(A)=\H^k(\car A)$. The topological support $\spt A$ is the closed support of its mass measure. All pointwise multiplicity statements will use $\car A$, not $\spt A$.

Products and the cone operator satisfy
\be\label{eq:chainrules}
\begin{aligned}
\partial(A\times B)&=(\partial A)\times B+A\times(\partial B),\\
\partial(p*A)&=A+p*(\partial A)\qquad (\dim A\geq1),
\end{aligned}
\ee
where in the last identity $p*A$ means the join of the $k$-current $A$ with the 0-current $\d_p$. It is a "cone" with vertex $p$. We also have
\be
(f\times g)_{\s}(A\times B)=(f_{\s}A)\times(g_{\s}B),\qquad
f_{\s}(A+B)=f_{\s}A+f_{\s}B.
\ee

\begin{lem}\label{lem:common}
Let $M\subset\R^n$ be $k$-rectifiable with locally finite $\H^k$ measure. Suppose that $A_1,\ldots,A_N\in\I_k(\R^n;\FF)$ satisfy $\car A_j\subset M$ up to null sets. Then there are measurable functions $\theta_j:M\to\FF$ with $A_j=\chain{M,\theta_j}$. Outside one common $\H^k$-null set, each carrier present at $x$ has approximate tangent plane $T_xM$, and
\be
\theta_{\sum_j\a_jA_j}(x)=\sum_j\a_j\theta_j(x)\quad\text{in }\FF.
\ee
\end{lem}
\nd
The rectifiable representation gives the $\FF$ multiplicity on each carrier; we extend it to the whole $M$ by zero on the rest of $M$. Two rectifiable sets of the same dimension have the same approximate tangent plane at almost every point of their intersection.
\qed

\begin{lem}\label{lem:pushmass}
Let $A\in\I_k(\R^n;\FF)$ have compact support, and let $p:\R^n\to\R^m$ be Lipschitz. Then
\be\label{eq:masspush}
\mass(p_{\s}A)\leq\int_{\car A}J_k(p|_{T_x\car A})\,d\H^k(x).
\ee
If $\car A\subset M$, one may replace $T_x\car A$ by $T_xM$ in the integrand.
\end{lem}
\nd

We apply the area formula, and get
\be \mass(p_{\s}A)\le \H^k(p(\car A))\le \int_{p(\car A)}\s p^{-1}\{y\}d\H^k(y)=\int_{\car A} J_k(p|_{T_x\car A})d\H^k(x).\ee\qed

\begin{lem}[Constancy]\label{lem:topcycle}
Let $P$ be an affine $k$-plane, $k\geq1$. If $A\in\I_k(P;\FF)$ is compactly supported and $\partial A=0$, then $A=0$.
\end{lem}
\nd
Choose $p\in P$ and look at the cone $p*A$. It is an integral $(k+1)$-current supported in the $k$ plane $P$, hence $p*A=0$. By \eqref{eq:chainrules}, $A=\partial(p*A)+p*(\partial A)=0$.\qed

\subsection{The tetrahedral cone and the classical paired calibrations}\label{subsec:classical-paired}

In $\R^3$, take
\be\label{eq:vertices}
\begin{aligned}
v_0&=\frac1{\sqrt3}(1,1,1),&v_1&=\frac1{\sqrt3}(1,-1,-1),\\
v_2&=\frac1{\sqrt3}(-1,1,-1),&v_3&=\frac1{\sqrt3}(-1,-1,1).
\end{aligned}
\ee
Let $\Delta=\conv\{v_0,v_1,v_2,v_3\}$ be the convex hull, and let $F_i$ be its facet opposite $v_i$. Let $K=\cup_{i<j}[v_i,v_j]$ be the 1-skeleton of the tetrahedron $\Delta$, and let $\T$ be the cone over $K$ centered at $0$. Set
\be\label{eq:T}
\Delta_{ij}=\conv\{0,v_i,v_j\},\qquad S=\T\cap\Delta=\bigcup_{i<j}\Delta_{ij}.
\ee
 Thus $S$ consists of these six truncated two-dimensional sectors $\Delta_{ij}, i<j$. Three sectors meet along each of the four rays $\R_+v_i$. Simple calculation gives

\begin{lem}\label{lem:geometry}
For $i\ne j$, $\sum_i v_i=0$ and $v_i\cdot v_j=-1/3$. For each $i$, The three sectors $\Delta_{ij}, j\ne i$ intersect at the ray $\R_+v_i$ at $120^\circ$. Moreover,
\be\label{eq:areas}
\H^2(\Delta_{ij})=\frac{\sqrt2}{3},\quad
\H^2(S)=2\sqrt2,\quad
q:=\H^2(F_i)=\frac2{\sqrt3},\quad w_i=-v_i,
\ee
where $w_i$ is the outward unit normal of $F_i$.
\end{lem}

Set $\Omega_i=\operatorname{int}\conv(\{0\}\cup F_i)$. If $\{i,j,k,l\}=\{0,1,2,3\}$, the interface between $\Omega_i$ and $\Omega_j$ is $\Delta_{kl}$. Then the unit normal pointing from $\Omega_i$ into $\Omega_j$ is
\be\label{eq:pairedfields}
\nu_{ij}=z_i-z_j\mbox{, where } z_i=\sqrt{\frac38}\,v_i.
\ee
$\nu_{ij}, i<j$ are just the Lawlor--Morgan paired calibrations for the $\T$ set \cite[Section~1.1]{LM94}.

\begin{figure}[htbp]
\centering
\begin{subfigure}[b]{.49\textwidth}\centering
\begin{tikzpicture}[scale=1.04,line cap=round]
\coordinate (v0) at (0,1.9);\coordinate (v1) at (-1.8,-.45);
\coordinate (v2) at (.35,-1.6);\coordinate (v3) at (1.8,.15);
\coordinate (o) at (.0875,0);
\fill[pale] (o)--(v1)--(v2)--cycle;
\foreach \i/\j in {0/1,0/2,0/3,1/2,2/3}\draw (v\i)--(v\j);
\draw[dashed] (v1)--(v3);
\foreach \i in {0,1,2,3}{\draw[inkblue,thick] (o)--(v\i);\fill (v\i) circle(1.6pt);}
\fill (o) circle(1.7pt);\node[right] at (o){$0$};
\node[above] at (v0){$v_0$};\node[left] at (v1){$v_1$};
\node[below] at (v2){$v_2$};\node[right] at (v3){$v_3$};
\node at (-.55,-.62){$\Delta_{12}$};
\end{tikzpicture}
\caption{The $K_4$ skeleton and one conical sector.}
\end{subfigure}\hfill
\begin{subfigure}[b]{.47\textwidth}\centering
\begin{tikzpicture}[scale=1.02,>=Latex]
\draw[very thick] (0,0)--(90:1.65);
\draw[very thick] (0,0)--(210:1.65);
\draw[very thick] (0,0)--(330:1.65);
\draw (90:.38) arc(90:210:.38);
\node at (150:1.05){$\Omega_1$};\node at (270:1.1){$\Omega_2$};\node at (30:1.1){$\Omega_3$};
\draw[->,inkblue,thick] (-.55,1.12)--(.55,1.12) node[right]{$\nu_{13}$};
\node[below] at (0,-1.5){$\nu_{13}=z_1-z_3$};
\node at (-.75,.1){$120^\circ$};
\fill (0,0) circle(1.8pt);
\end{tikzpicture}
\caption{A section orthogonal to a triple ray.}
\end{subfigure}
\caption{The paired calibration of the tetrahedral cone. The left panel is schematic and does not preserve angles. In the right panel, the three sector traces meet at $120^\circ$.}
\label{fig:paired}
\end{figure}
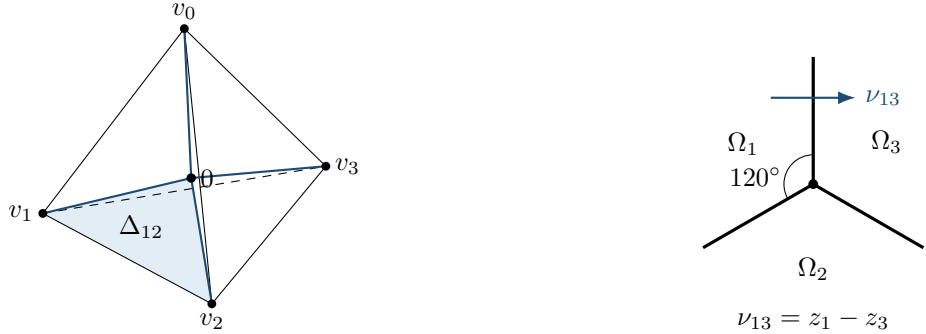

Now let us give the idea of using the paired calibrations to prove the Almgren minimality for $\T$. Since $\T$ is a cone, it is enough to prove its minimality in the domain $\mathring\Delta$. 

First let $\widetilde\Omega_0,\ldots,\widetilde\Omega_3$ be a finite-perimeter partition of $\Delta$ with boundary traces on the respective facets $F_i$. Set
$M_{ij}=\partial^*\widetilde\Omega_i\cap\partial^*\widetilde\Omega_j\cap\operatorname{int}\Delta$, and orient its unit normal $\nu_{ij}^M$ from region $i$ toward region $j$. Apply Stoke's formula to each region and sum over $i$, and we get
\begin{align}\label{eq:pairedflux}
\H^2(\T\cap \Delta)=2\sqrt2
&=-\sum_i\int_{F_i}z_i\cdot w_i\,d\H^2
=-\sum_i\left[\sum_{j\ne i}\int_{M_{ij}}z_i\cdot \nu_{ji}^M\,d\H^2\right] \\
&=\sum_{i<j}\int_{M_{ij}}(z_i-z_j)\cdot\nu_{ij}^M\,d\H^2
\leq\sum_{i<j}\H^2(M_{ij}).\notag
\end{align}

As a result, if $M$ is a closed $2$-rectifiable set in $\Delta$ that separates the 4 facets $F_i^\circ=F_i\bs \T, 0\le i\le 3$ (this means that $M\cap \pa\Delta=\T\cap \pa\Delta$, and the $F_i^\circ,0\le i\le 3$ belong to difference connected components of $\Delta\bs M$), then let $\widetilde\Omega_i,0\le i\le 2$ be the 3 connected components of $\Delta\bs M$ that contains $F_i^\circ$, and set $\widetilde\Omega_3=\Delta\bs \overline{\cup_{0\le i\le 2}\widetilde\Omega_i}$. Then the sets $\widetilde\Omega_i,0\le i\le 3$ forms a finite-perimeter partition of $\Delta$ as above. By definition, we know that the above defined $M_{ij}, 0\le i<j\le 3$ are disjoint and contained in $M$. Thus (\ref{eq:pairedflux}) yields that $\H^2(M)\ge \H^2(\T\cap \Delta)$.

As a result, $\T\cap \Delta$ minimizes the Hausdorff measure among all closed $2$-rectifiable sets in $\Delta$ that separates the 4 facets $F_i^\circ$. It is easy to check that this class of separators is larger than the class of all Almgren competitors of $\T\cap \Delta$ in $\Delta$ (see for example \cite{topo}). We then get the Almgren minimality for $\T$.\qed

\medskip

\textbf{Why the product of paired calibrations for $\T$ cannot be used to prove the minimality for the product $\T\times \T$:}

In the key argument of paired calibrations above, the word ``paired'' refers to differences of vector fields on adjacent regions, and the flux computation proves the comparison for the stated partition class. 

Codimension-one separation supplies adjacent-region information; it cannot simply be assumed for a codimension-two image: at a typical point of a set of codimension two, there is no natural way to control the number of regions meeting there. Different pieces may overlap in complicated ways, and the combined contribution of the associated product calibration fields can no longer be controlled by the classical argument. This is the overlap obstruction explained for $\Y\times\Y$ in \cite[Section~1.3.2]{YXY}.

For the product $\T\times \T$, a natural approach is to construct a suitable intrinsic decomposition of the set, so as to control the overlap combined contribution of the calibration vector fields at each point. This is precisely the approach used by the second author in her previous paper \cite{YXY}, using $\FF$ as the coefficient group. Let $E_i, 0\leq i\leq2$, denote the three facets of the triangular convex hull associated with $\Y$, and define $E'_j$ similarly for $\Y'$. From each competitor for $\Y\times \Y'$, the second author extracted nine pieces homologous to the nine products $E_i\times E'_j$, respectively. The combined contribution of the corresponding products of paired calibrations can then be controlled pointwise, even where these pieces overlap. This yields the minimality of $\Y\times \Y'$. But this facet-based decomposition does not directly provide the desired control over overlaps for products involving the tetrahedral cone $\T$ or any regular simplex skeleton cone in higher dimensions, although these cones can be proved minimal by a similar paired-calibration argument (see the work of Lawlor\&Morgan \cite{LM94} and Brakke \cite{Br91}).

Our key idea for dealing with this issue, is to replace the facet-based family by the full family of $\FF$ cycles. This gives many more test currents, but in return their multiplicities satisfy a much more symmetric algebraic rule, so that we can control the overlap algebraically, rather than using the separation condition. In the comming section, we explain this construction first for the tetrahedral cone.

\section{The complete cycle family for the tetrahedral cone}\label{sec:full-cycle}

In this section we develop the complete $\FF$ cycle family for
$\T$. We first introduce the complete family and the spanning competitor
class in Subsection~\ref{subsec:tetra-cycles}, then construct the
associated projections and area vectors in Subsection~\ref{subsec:area-vectors}. In
Subsection~\ref{subsec:tetra-frame} we prove the symmetry identity
and the pointwise estimate. Finally, in
Subsection~\ref{subsec:tetra-minimality} we combine these results
to give a new proof of the Almgren minimality of $\T$.

\subsection{The complete $\FF$ cycle family}\label{subsec:tetra-cycles}

Take all the notations as in Subsection \ref{subsec:classical-paired}. Let $K=\bigcup_{i<j}[v_i,v_j]$ be the one-skeleton of $\Delta$, viewed as the graph $K_4$. Write $e_{ij}=e_{ji}$ for its unoriented edge with vertices $v_i$ and $v_j$. Thus the chain spaces over $K$ are
\be
C_1(K;\FF)=\bigoplus_{i<j}\FF e_{ij},\quad
C_0(K;\FF)=\bigoplus_i\FF v_i,\quad
\partial e_{ij}=v_i+v_j.
\ee
A one-chain is a cycle exactly when the number of selected edges at every vertex is even. 

The four triangular cycles $\partial\chain{F_i}$ are the generators of $H_1(K;\FF)$ and satisfy
\be\label{eq:tri_relation}
\sum_{i=0}^3\partial\chain{F_i}=0.
\ee
This is their only relation. Thus
\be\label{eq:cycle_dim}
H_1(K;\FF)\cong \FF^3\mbox{ and }\dim_\FF H_1(K;\FF)=3.
\ee

Let $G=\FF^3$, with basis $e_1=(1,0,0),e_2=(0,1,0),e_3=(0,0,1)$. For $g=(g_1,g_2,g_3)\in G$, define
\be\label{eq:chains}
P_g=\sum_{i=1}^3g_i\chain{F_i},\qquad
\lambda_g=\partial P_g,\qquad D_g=0*\lambda_g.
\ee
The current $P_g$ is a two-dimensional facet chain,
$\lambda_g$ is a one-cycle in the skeleton, and $D_g$ is its
two-dimensional conical filling. All three assignments are
$\FF$-linear in $g$, and
\be\label{eq:D_boundary}
\partial D_g=\lambda_g,\qquad\spt D_g\subset S.
\ee
The map $g\mapsto\lambda_g$ identifies $G$ with $H_1(K;\FF)$, because $\partial\chain{F_1},\partial\chain{F_2},\partial\chain{F_3}$ form a basis.

\begin{table}[htbp]\centering\small
\renewcommand{\arraystretch}{1.3}
\begin{tabular}{c c l c}
\toprule
$g$&Facets selected in $P_g$&Edge sum for $\lambda_g$&Type\\\midrule
$100$&$F_1$&$e_{02}+e_{23}+e_{30}$&Triangle\\
$010$&$F_2$&$e_{01}+e_{13}+e_{30}$&Triangle\\
$001$&$F_3$&$e_{01}+e_{12}+e_{20}$&Triangle\\
$110$&$F_1,F_2$&$e_{02}+e_{23}+e_{31}+e_{10}$&Quadrilateral\\
$101$&$F_1,F_3$&$e_{03}+e_{32}+e_{21}+e_{10}$&Quadrilateral\\
$011$&$F_2,F_3$&$e_{03}+e_{31}+e_{12}+e_{20}$&Quadrilateral\\
$111$&$F_1,F_2,F_3$&$e_{12}+e_{23}+e_{31}=\partial\chain{F_0}$&Triangle\\\bottomrule
\end{tabular}
\caption{All seven nonzero group elements.}\label{tab:cycles}
\end{table}

\begin{figure}[htbp]\centering
\begin{tikzpicture}[scale=1, line cap=round]
\foreach \s in {0,4.3,8.6}{
\begin{scope}[xshift=\s cm]
\coordinate (a0) at (0,1);\coordinate (a1) at (-1,0);\coordinate (a2) at (0,-1);\coordinate (a3) at (1,0);
\foreach \i/\j in {0/1,0/2,0/3,1/2,1/3,2/3}\draw[gray!40] (a\i)--(a\j);
\foreach \i in {0,1,2,3}\fill (a\i) circle(1.7pt);
\node[above] at (a0){$0$};\node[left] at (a1){$1$};\node[below] at (a2){$2$};\node[right] at (a3){$3$};
\end{scope}}
\begin{scope}\draw[very thick,inkblue] (0,1)--(0,-1)--(1,0)--cycle;\node at (0,-1.65){$\lambda_{100}$};\end{scope}
\begin{scope}[xshift=4.3cm]\draw[very thick,inkblue] (0,1)--(-1,0)--(1,0)--cycle;\node at (0,-1.65){$\lambda_{010}$};\end{scope}
\begin{scope}[xshift=8.6cm]\draw[very thick,inkblue] (0,1)--(0,-1)--(1,0)--(-1,0)--cycle;\node at (0,-1.65){$\lambda_{110}$};\end{scope}
\node at (2.15,0){$+$};\node at (6.45,0){$=$};
\end{tikzpicture}
\caption{Addition modulo two is symmetric difference of edge sets. In the above  $\lambda_{100}+\lambda_{010}=\lambda_{110}$, the common edge $e_{03}$ in the first two cycles occurs twice and cancels. }
\label{fig:cycles}
\end{figure}
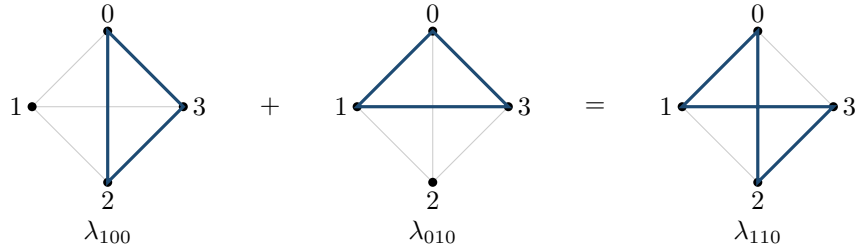

\begin{lem}\label{lem:relative}
The boundary map induces natural isomorphisms
\be
H_2(S,K;\FF)\simeq H_1(K;\FF)\simeq G,
\ee
under which $[D_g]$ corresponds to $[\lambda_g]$.
\end{lem}
\nd
The set $S$ is the cone over a graph and hence is contractible. In the long exact sequence of $(S,K)$, both $H_2(S;\FF)$ and $H_1(S;\FF)$ vanish, so the connecting boundary map is an isomorphism. 
\qed

The currents $D_g$ represent the $\FF$ topology that will be preserved under deformation. 

As we can see in the Table \ref{tab:cycles}, the four $\lambda_g$ for triangular type $g$ correspond naturally to the four facet-boundary cycles $\partial\chain{F_i}$ appearing in the classical paired calibration. The remaining three nonzero $\lambda_g$ are quadrilateral cycles. 

\begin{defn}[Spanning competitor]\label{spanningT} Let $\mathscr C(\Delta)$ be the class of all compact sets $F\subset \Delta$ so that $F$ is 2-rectifiable, $\H^2(F)<\infty$, and there exist three currents $\Gamma_r\in \I_2(\R^3;\FF), 1\le r\le 3$, so that 
\be \spt(\Gamma_r)\subset F, \pa \Gamma_r=\lambda_{e_r}.\ee
For every $g=(g_1,g_2,g_3)\in G$, set $\Gamma_g=\sum_{r=1}^3 g_r\Gamma_r$. Then by linearity over $\FF$, we know that 
\be \pa\Gamma_g=\lambda_g=\pa D_g.\ee
\end{defn}

Next, let us look at the pointwise multiplicity of each part $\Gamma_g, g\in \Gx$ on each spanning competitor. For a given $F\in \mathscr C(\Delta)$, and the corresponding $\Gamma_r, 1\le r\le 3$, we apply Lemma \ref{lem:common} on their common carrier $F$, and get the $\FF$-multiplicity $\theta_r: F\to \FF, 1\le r\le 3$, so that $\Gamma_r=\chain{F,\theta_r}$. At every $x\in F$, set 
\be c(x)=(\theta_1(x),\theta_2(x),\theta_3(x))\in G.\ee Then we have
\be \theta_{\Gamma_g}(x)=g\cdot c(x)\mbox{, and }x\in \car \Gamma_g\Leftrightarrow g\cdot c(x)=1.\ee
Next, we will associate with each $g$, an orthogonal projection and an area vector (which takes the role of  calibration).

\subsection{Projection and area vectors}\label{subsec:area-vectors}
Denote the standard Euclidean basis by $\mathbf e_i$, to distinguish it from the $\FF$ basis $e_i$ of $G$. Set $\Gx=G\bs\{0\}$. For $g\in\Gx$, define the real area vector
\be\label{eq:ag}
a_g=q\sum_{i=1}^3 g_iw_i\mbox{, and set } A_g=|a_g|.
\ee
That is, $a_g$ is the sum of the oriented area vectors of the selected facets. Note that although $g\mapsto D_g$ is linear over $\FF$, the map
$g\mapsto a_g$ is not additive with respect to addition modulo two.

The currents $D_g$ record the topology, while the real vector $a_g$ records the size of the corresponding projection:

\begin{lem}\label{lem:polygon}
For each $g\in \Gx$, let $\rho_g:\R^3\to a_g^\perp$ be the orthogonal projection along the area vector $a_g$. Then there is a convex polygon $Q_g\subset a_g^\perp$, with the origin in its relative interior, such that
\be\label{eq:polygon}
(\rho_g)_{\s}D_g=\chain{Q_g},\qquad\H^2(Q_g)=A_g.
\ee
More precisely, for the four triangle type $g$, $Q_g$ is an equilateral triangle of area $2/\sqrt3$, and the three $g$ of quadrilateral type give squares of area $4/3$.
\end{lem}

\begin{table}[H]\centering
\renewcommand{\arraystretch}{1.3}
\begin{tabular}{c c c c}
\toprule
$g$&$a_g$&$A_g$&Projection polygon\\\midrule
$100$&$\frac23(-1,1,1)$&$2/\sqrt3$&Equilateral triangle\\
$010$&$\frac23(1,-1,1)$&$2/\sqrt3$&Equilateral triangle\\
$001$&$\frac23(1,1,-1)$&$2/\sqrt3$&Equilateral triangle\\
$110$&$\frac43(0,0,1)$&$4/3$&Square\\
$101$&$\frac43(0,1,0)$&$4/3$&Square\\
$011$&$\frac43(1,0,0)$&$4/3$&Square\\
$111$&$\frac23(1,1,1)$&$2/\sqrt3$&Equilateral triangle\\\bottomrule
\end{tabular}
\caption{Area vectors. Their directions are normal to the projection planes, and their lengths equal the areas of the projection polygons.}\label{tab:vectors}
\end{table}

\nd
\underline{Triangle type:} 
For $g=e_i$, $a_g=-qv_i$, so the projection plane is parallel to $F_i$. For $j\ne i$, since $v_i\cdot v_j=-1/3$,
$\rho_g(v_j)=v_j+(1/3)v_i$. Since the centre of $F_i$ is $-v_i/3$, $\rho_g$ translates $F_i$ to a triangle centered at the origin without changing its area. The three sectors of $D_g$ project onto the central triangulation of $\rho_g(F_i)$ joining the center to the three edges respectively; the interiors are disjoint and each sector maps affinely and bijectively onto a nondegenerate triangle (see Figure~\ref{fig:polygons}(a)). The projected current therefore has mod-two multiplicity one almost everywhere in the polygon. For $g=111$, we have $\lambda_g=\partial\chain{F_0}$ and $a_g=qv_0$, and the same argument applies to $F_0$.

\underline{Quadrilateral type:} Take $g=110$ for example, $a_g=(4/3)\mathbf e_3$. Along the cycle $v_0,v_2,v_3,v_1,v_0$, the projected vertices are, in order,
\be
\frac1{\sqrt3}(1,1),\quad\frac1{\sqrt3}(-1,1),\quad
\frac1{\sqrt3}(-1,-1),\quad\frac1{\sqrt3}(1,-1).
\ee
They form a square of side $2/\sqrt3$. The four conical triangles project onto its central triangulation (see Figure~\ref{fig:polygons}(b)). Thus we again obtain a current identity, not just coverage of the support. The vector lengths in Table~\ref{tab:vectors} equal the asserted areas.
\qed

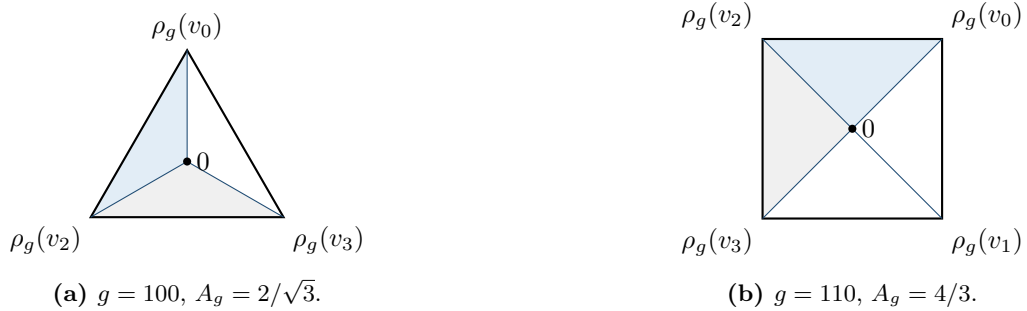
\begin{figure}[htbp]\centering
\begin{subfigure}{.45\textwidth}\centering
\begin{tikzpicture}[scale=.95]
\coordinate (o) at (0,0);\coordinate (a) at (90:1.55);\coordinate (b) at (210:1.55);\coordinate (c) at (330:1.55);
\fill[pale] (o)--(a)--(b)--cycle;\fill[gray!12] (o)--(b)--(c)--cycle;
\draw[thick] (a)--(b)--(c)--cycle;\foreach \x in {a,b,c}\draw[inkblue] (o)--(\x);
\fill (o) circle(1.6pt);\node[right] at (o){$0$};
\node[above] at (a){$\rho_g(v_0)$};\node[below left] at (b){$\rho_g(v_2)$};\node[below right] at (c){$\rho_g(v_3)$};
\end{tikzpicture}\caption{$g=100$, $A_g=2/\sqrt3$.}
\end{subfigure}\hfill
\begin{subfigure}{.45\textwidth}\centering
\begin{tikzpicture}[scale=.95]
\fill[pale] (0,0)--(-1.25,1.25)--(1.25,1.25)--cycle;
\fill[gray!12] (0,0)--(-1.25,1.25)--(-1.25,-1.25)--cycle;
\draw[thick] (-1.25,-1.25) rectangle(1.25,1.25);
\foreach \x/\y in {-1.25/-1.25,-1.25/1.25,1.25/-1.25,1.25/1.25}\draw[inkblue] (0,0)--(\x,\y);
\fill (0,0) circle(1.6pt);\node[right] at (0,0){$0$};
\node[above right] at (1.25,1.25){$\rho_g(v_0)$};\node[above left] at (-1.25,1.25){$\rho_g(v_2)$};
\node[below left] at (-1.25,-1.25){$\rho_g(v_3)$};\node[below right] at (1.25,-1.25){$\rho_g(v_1)$};
\end{tikzpicture}\caption{$g=110$, $A_g=4/3$.}
\end{subfigure}
\caption{The projection of $D_g$ covers $Q_g$ with multiplicity one modulo two. }
\label{fig:polygons}
\end{figure}

Now we prove that the projections for $\Gamma_g, g\in \Gx$ for every spanning competitor $F\in \mathscr C(\Delta)$ are still the same as those of $D_g$.

\begin{lem}\label{lem:projectionsum}Let $F\in \mathscr C(\Delta)$, and let $\Gamma_g$ be as in Definition \ref{spanningT}. Then for every $g\in \Gx$, 

$1^\circ$ $(\rho_g)_\s\Gamma_g=\chain{Q_g}$.

$2^\circ$ \be A_g^2\le \int_{\car \Gamma_g}|a_g\cdot n_x|d\H^2(x)=\int_F\mathbf1_{g\cdot c(x)=1}|a_g\cdot n_x|d\H^2(x),\ee where $n_x$ is a unit normal to $T_xF$ defined almost everywhere. As a result, we have
\be \sum_{g\in \Gx}A_g^2\le\int_{F}\left[\sum_{g\in\Gx, g\cdot c(x)=1}|a_g\cdot n_x|\right]d\H^2(x)\ee
\end{lem}

\nd $1^\circ$ By Lemma \ref{lem:polygon}, we know that $(\rho_g)_\s D_g=\chain{Q_g}$. Hence the current $Z=(\rho_g)_\s\Gamma_g+\chain{Q_g}$ satisfies that 
\be \pa Z=\pa (\rho_g)_\s(\Gamma_g)+\pa(\rho_g)_\s D_g=(\rho_g)_\s(\pa\Gamma_g+\pa D_g)=0.\ee
Thus the constancy Lemma \ref{lem:topcycle} yields $Z=0$, which gives the conclusion.

$2^\circ$ Directly by definition we know that at almost every point $x\in F$, 
\be J_2(\rho_g|_{T_xF})=\left|\frac{a_g}{A_g}\cdot n_x\right|.\ee
Thus by Lemmas \ref{lem:pushmass} and $1^\circ$, we have
\be \begin{split}A_g&=\mass(\chain{Q_g})=\mass((\rho_g)_\s\Gamma_g)\le\int_{\car \Gamma_g}J_2(\rho_g|_{T_xF})d\H^2(x)\\
&=\int_{\car \Gamma_g}A_g^{-1}|a_g\cdot n_x|d\H^2(x)=A_g^{-1}\int_F\mathbf1_{g\cdot c(x)=1}|a_g\cdot n_x|d\H^2(x).
\end{split}\ee
Multiply both sides by $A_g$, and we get $2^\circ$.
\qed

\subsection{Symmetry of the area vectors and the pointwise estimate}\label{subsec:tetra-frame}

In this section, we are going to control $\sum_{g\in\Gx, g\cdot c(x)=1}|a_g\cdot n_x|$ at every point. In fact, we are going to give a uniform control on all $c\in G$ and all unit vector $n\in \R^3$:

\begin{thm}[Pointwise estimate]\label{thm:pointwiseT}For every $c\in G$ and every unit vector $n\in \R^3$, the function
\be \Psi(c,n):=\sum_{g\in\Gx, g\cdot c=1}|a_g\cdot n|\le 2\sqrt\kappa=\frac{8\sqrt 2}{3}.\ee
The constant is optimal.
\end{thm}

The proof of the above theorem relies on the following lemma on the symmetry of the whole family of area vectors. This is the quantitative reason for keeping all nonzero elements of $G$.

\begin{lem}[Symmetry identity]\label{lem:frame}
Set $\kappa=32/9$. Then
\be\label{eq:frame}
\sum_{g\in\Gx}a_ga_g^T=\kappa I_3,\qquad
\sum_{g\in\Gx}A_g^2=3\kappa=\frac{32}{3}.
\ee
Equivalently, for every $w\in\R^3$,
\be\label{eq:l2frame}
\sum_{g\in\Gx}|a_g\cdot w|^2=\kappa|w|^2.
\ee
\end{lem}
\nd For the 4 triangular type $g$, the corresponding vectors $a_g$ are just $\pm 2/\sqrt 3$ times the four vectors $v_i,0\le i\le 3$. Their outer products sum to a matrix with diagonal entries 4 and off-diagonal entries 0, so their contribution is $16I_3/9$. The three square vectors are $(4/3)\mathbf e_i$ and give another $16I_3/9$. Adding gives the first identity, taking the trace gives the second. Evaluating the quadratic form on $w$ gives \eqref{eq:l2frame}.
\qed

\begin{rem}\label{rem:2.28}
There are two different symmetries involved here.

Geometric symmetry: the Euclidean symmetry group of the tetrahedron permutes
the four triangular cycles and the three quadrilateral cycles.
These form two different orbits and have different lengths. Since outer products are
unchanged by signs, the sum of squares has no preferred direction:
this is the geometric meaning of the symmetry identity.

Algebraic symmetry: since we take the complete family $\Gx$, every nonzero linear functional
$g\longmapsto g\cdot c, c\ne0$ takes the value one on exactly four elements of $G^\times$: 
this is the uniform four-point support used in the proof of
Theorem~\ref{thm:pointwiseT}. 

Notice that the four triangular area vectors alone already satisfy
a symmetry identity, with constant $16/9$. So the aim of adding the quadrilaterals is not to repair any lack of symmetry. Their role is to complete the triangular family to $G^\times$, and hence to obtain the uniform algebraic support property above. Its importance for the product argument will become clear in Remark~\ref{rem:triangle-product}.
\end{rem}

\noindent\textbf{Proof of Theorem \ref{thm:pointwiseT}:}
If $c=0$, the sum is zero. Otherwise, $g\mapsto g\cdot c$ is a nonzero linear functional on $G=\FF^3$. Hence its value 0 and 1 level sets each contain exactly four elements. Since all the elements in its value 1 set are nonzero, hence are in $\Gx$. As a result, by Cauchy-Schwarz and Lemma \ref{lem:frame},
\be \Psi(c,n)\le 2\left(\sum_{g\in\Gx, g\cdot c=1}|a_g\cdot n|^2\right)^{1/2}\le  2\left(\sum_{g\in\Gx}|a_g\cdot n|\right)^{1/2}=2\sqrt\kappa.\ee

To prove that the constant is attained, just take $c=011$ and $n=\frac{1}{\sqrt 2}(0,1,-1)$. The four active elements $g$ so that $g\cdot c=1$ are $010,001,110,101$, and Table \ref{tab:vectors} gives $|a_g\cdot n|=\frac{2\sqrt 2}{3}$ for each of these $g$. Their sum is $\frac{8\sqrt 2}{3}=2\sqrt\kappa$.\qed

\subsection{Minimality of $\T$}\label{subsec:tetra-minimality}

Combine the above results, we first prove the minimality of $\T$ in the spanning competitor class, and then deduce its Almgren minimality.

\begin{thm}[Minimality of $\T$]\label{thm:T-fullcycle} 
$1^\circ$ $\H^2(\T\cap\Delta)=\inf_{F\in \mathscr C(\Delta)}\H^2(F)$.

$2^\circ$ $\T$ is Almgren minimal in $\R^3$.
\end{thm}

\nd $1^\circ$ Take any $F\in  \mathscr C(\Delta)$. Let $\Gamma_g$ be as in Definition \ref{spanningT}. Then Lemmas \ref{lem:projectionsum}, \ref{lem:frame} and Theorem \ref{thm:pointwiseT} give
\be  3\kappa=\sum_{g\in \Gx}A_g^2\le\int_{F}\left[\sum_{g\in\Gx, g\cdot c(x)=1}|a_g\cdot n_x|\right]d\H^2(x)=\int_{F}\Psi(c(x),n_x)d\H^2(x)\le 2\sqrt\kappa \H^2(F). \ee
That is, 
\be\inf_{F\in \mathscr C(\Delta)}\H^2(F)\ge \frac 32\sqrt\kappa.\ee
Now the currents $D_{e_i},1\le i\le 3$ show that $\T\cap \Delta\in \mathscr C(\Delta)$. Lemma \ref{lem:geometry} tells that $\H^2(\T\cap \Delta)=2\sqrt 2=\frac 32\sqrt\kappa$. This gives $1^\circ$.

$2^\circ$  Since $\T$ is a cone, it suffices to prove its Almgren minimality in the domain $\mathring\Delta$.  

Let $\{\varphi_t\}$ be an arbitrary Almgren deformation in $\mathring\Delta$, and set $f=\varphi_1$. Set $F:=f(E\cap\Delta)$. Then $F$ is compact, 2-rectifiable, and has finite $\H^2$-measure.

Define $\Gamma_{r}=f_{\s}(D_{e_r})$. Since $\spt\partial D_{e_r}
=\spt\lambda_{e_r}\subset\partial\Delta$ and $f$ is the identity near $\partial\Delta$,
%
\be
\partial\Gamma_r
=f_{\s}\partial(D_{e_r})
=\partial(D_{e_r})=\lambda_{e_r}.
\ee
Hence $F\in\mathscr C(\Delta)$, and $1^\circ$ gives
\be
\H^4(F)\ge\H^4(\T\cap\Delta).
\ee
Subtracting the common part, we obtain
\be
\H^4((\T\cap\Delta)\bs f(\T\cap\Delta))\leq\H^4(f(\T\cap\Delta)\bs (\T\cap\Delta)).\ee
\qed

\section{Product multiplicities and the pointwise estimate in the tetrahedral case}\label{sec:estimate}

We now pass from one tetrahedral cone to the product. The first problem is to describe what happens when different product faces overlap after a noninjective deformation. The main point is that all local multiplicities can be encoded by a single $\FF$ matrix.

\subsection{Product coefficients and the nine-dimensional tensor space}

Now let $X\cong X'\cong\R^3$, use all the notations of Subsection \ref{subsec:classical-paired} and Section \ref{sec:full-cycle} for $X$, and the same with prime in $X'$. We will look at the product $\T\times \T'$ in the product space $X\times X'$. We fix bases in both cycle spaces, identify $G'$ with $G$ in coordinates, and write the second index as $h\in G$. Set
\be
E=\T\times\T',\quad E_0=S\times S',\quad
\Dcal=\Delta\times\Delta',\quad
\mathcal B=(K\times S')\cup(S\times K').
\ee
Here $S,S'$ are two-dimensional, $E_0$ is four-dimensional, $\Dcal$ is a six-dimensional bounded product domain, and $\mathcal B\subset\partial\Dcal$ is the three-dimensional boundary complex of $S\times S'$. 

\begin{lem}\label{lem:producthomology}
The cross product induces a natural isomorphism
\be\label{eq:producthomology} H_4(E_0,\mathcal B;\FF)
\simeq H_2(S,K;\FF)\otimes_\FF H_2(S',K';\FF).\ee
Under the chosen identifications $H_2(S,K;\FF)\simeq G$ and $H_2(S',K';\FF)\simeq G'$, this gives an isomorphism with $G\otimes G'$.

More precisely, the nine classes $[D_{e_r}\times D'_{e_s}]$, $1\leq r,s\leq3$, form a basis, and for every $(g,h)\in G\times G'$, $[D_g\times D'_h]$ corresponds to the pure tensor $g\otimes h$.
\end{lem}
\nd
The pairs $(S,K)$ and $(S',K')$ are finite CW pairs, and their
product pair is
\be
(E_0,\mathcal B)
=
\bigl(S\times S',\,
      (K\times S')\cup(S\times K')\bigr).
\ee
Since $\FF$ is a field, the relative K\"unneth theorem
\cite[\S3.B, p.~276]{Hat} gives a natural isomorphism
\be\label{eq:kunneth}
\bigoplus_{\substack{p,q\geq0\\p+q=4}}
H_p(S,K;\FF)\otimes_{\FF}H_q(S',K';\FF)
\xrightarrow{\;\cong\;}
H_4(E_0,\mathcal B;\FF),
\ee
induced by the homology cross product. There are no
$\operatorname{Tor}$ terms over $\FF$.

Since $S$ and $S'$ have dimension two, their relative cellular chain groups vanish in degrees greater than two, and hence
\be
H_j(S,K;\FF)=H_j(S',K';\FF)=0
\qquad\text{for }j>2.
\ee
Consequently, among the summands with $p+q=4$, only the term
$(p,q)=(2,2)$ can be nonzero. Thus \eqref{eq:kunneth} restricts
to a natural isomorphism
\be
\begin{aligned}
\mathcal P:\quad
H_2(S,K;\FF)\otimes_{\FF}H_2(S',K';\FF)
&\longrightarrow H_4(E_0,\mathcal B;\FF),\\
[Q]\otimes[Q']&\longmapsto[Q\times Q'].
\end{aligned}
\ee
This gives the isomorphism \eqref{eq:producthomology} .
%
%
\qed

A general tensor of $G\otimes G'$ has the form
\be\label{eq:tensor_M}
z=\sum_{r,s=1}^3 M_{rs}e_r\otimes e'_s,\qquad M\in M_3(\FF).
\ee
This nine-dimensional space contains $2^9=512$ elements. Among these elements, there are 49 nonzero simple tensors $g\otimes h$, which correspond to the 49 rank-one matrices $gh^T$. They are distinct because the only nonzero scalar in $\FF$ is one. Besides the zero tensor and the 49 nonzero rank-one tensors, the remainings are rank-two or rank-three tensors, and are sums of the above simple tensors, rather than single geometric products.

For the corresponding currents, the product boundary formula gives
\be\label{eq:productboundary}
\partial(D_g\times D'_h)
=
\lambda_g\times D'_h+D_g\times\lambda'_h,
\qquad
\spt\partial(D_g\times D'_h)\subset\mathcal B.
\ee
Since $\mathcal B\subset\partial\Dcal$, these boundary currents
are unchanged by Almgren deformations supported in
$\operatorname{int}\Dcal$. This will allow us to apply the
fixed-boundary comparison in Section~\ref{sec:minimality}.

\subsection{Local multiplicities and incidence matrices}
Let $F\subset\R^6$ be countably 4-rectifiable with finite measure, and let $\Gamma_{rs}, 1\le r,s\le 3$ be nine compactly supported $\FF$-integral currents whose essential carriers lie in $F$. By Lemma~\ref{lem:common}, write
\be\label{eq:C}
\Gamma_{rs}=\chain{F,\theta_{rs}}.
\ee
For every $x\in F$, let 
\be\label{eq:multiplicityC} C(x)=(\theta_{rs}(x))_{r,s=1}^3\in M_3(\FF)\ee
be the matrix of multiplicities on $x$.

Now for the 49 simple tensors $g\otimes h$, $g,h\in \Gx$, define
\be\label{eq:49}
g\otimes h\mapsto\Gamma_{g,h}=\sum_{r,s=1}^3g_rh_s\Gamma_{rs}=\chain{F,\sum_{r,s=1}^3g_rh_s\theta_{rs}}=\chain{F,\tr(C(x)^T(g\otimes h))}.\ee

This determines a linear map from $G\otimes G'$ into the space of currents. For the general tensor \eqref{eq:tensor_M}, the multiplicity of the associated current at $x$ is
\be
\sum_{r,s}M_{rs}\theta_{rs}(x)=\tr(C(x)^TM)\quad\text{in }\FF.
\ee
Intrinsically, $C(x)$ represents a linear functional in $(G\otimes G')^*$.

\begin{pro}\label{prop:parity} Suppose $g,h\in\Gx$.

$1^\circ$ Let $\theta_{g,h}$ be the $\FF$-multiplicity of $\Gamma_{g,h}$ on $F$. Then outside one common $\H^4$-null set, 
\be\label{eq:parity}
\theta_{g,h}(x)=g^TC(x)h,
\mbox{ and }
x\in\car\Gamma_{g,h}\ \Longleftrightarrow\ g^TC(x)h=1.
\ee

$2^\circ$ If $f$ is a Lipschitz map, and $\Gamma_{rs}=f_{\s}(D_{e_r}\times D'_{e_s})$, then one also has the exact current identity
\be\label{eq:pushproduct}
\Gamma_{g,h}=f_{\s}(D_g\times D'_h).
\ee
\end{pro}
\nd
The first equality follows directly from \eqref{eq:49}. Nonzero multiplicity is equivalent to membership in the essential carrier. Finally, expand $D_g=\sum_rg_rD_{e_r}$ and $D'_h=\sum_sh_sD'_{e_s}$, apply bilinearity of the product and linearity of pushforward,  and we get \eqref{eq:pushproduct}.
\qed

For example, when $g=110$ and $h=101$, one has $g^TCh=C_{11}+C_{13}+C_{21}+C_{23}$. The multiplicity is one if the four selected positions contain an odd number of ones, and zero otherwise. Geometrically, this is just an addition of currents modulo two. A noninjective map can fold several source faces onto the same target face and the folded currents may cancel because of an even number of coverings.

\begin{figure}[htbp]\centering
\begin{tikzpicture}[x=.95cm,y=.95cm]
\foreach \r in {1,2,3}\foreach \s in {1,2,3}{
\draw[gray!60] (\s-1,-\r+1) rectangle(\s,-\r);}
\foreach \r in {1,2}\foreach \s in {1,3}{\fill[pale] (\s-1,-\r+1) rectangle(\s,-\r);\draw[inkblue,thick] (\s-1,-\r+1) rectangle(\s,-\r);}
\foreach \r/\s/\v in {1/1/1,1/2/0,1/3/1,2/1/0,2/2/1,2/3/1,3/1/1,3/2/0,3/3/0}\node at (\s-.5,-\r+.5){$\v$};
\node[left] at (0,-1.5){$C(x)=$};
\node[above,font=\scriptsize] at (.5,.05){$h_1=1$};\node[above,font=\scriptsize] at (1.5,.05){$h_2=0$};\node[above,font=\scriptsize] at (2.5,.05){$h_3=1$};
\node[right] at (3.15,-.5){$g_1=1$};\node[right] at (3.15,-1.5){$g_2=1$};\node[right] at (3.15,-2.5){$g_3=0$};
\node[align=left,anchor=west,font=\small] at (5,-1.5){Four positions are selected;\\three have multiplicity $1$.\\$1+1+0+1=1\quad(\bmod 2)$};
\end{tikzpicture}
\caption{Reading the local multiplicity matrix. It records nine basic currents at one point.}
\label{fig:parity}
\end{figure}

To describe the multiplicities on the six sectors of $S=\T\cap \Delta$ uniformly, set
\be
c_0=0,\qquad c_i=e_i,\quad 1\leq i\leq3.
\ee
Whenever $\{i,j,p,q\}=\{0,1,2,3\}$, define
\be\label{eq:alpha}
\a_{ij}=c_p+c_q.
\ee

The edge $[v_i,v_j]$ belongs exactly to $F_p,F_q$, so the multiplicity of $D_g$ on the relative interior $\mathring{\Delta_{ij}}$ of $\Delta_{ij}$ is $g\cdot\a_{ij}$. Thus, in the relative interior of a product face of $S\times S'$,
\be\label{eq:rankoneoriginal}
C(x)=\a_{ij}(\a'_{kl})^T,
\qquad x\in(\mathring{\Delta_{ij}}\times\mathring{\Delta'_{kl}}).
\ee
After a general pushforward of $S\times S'$, $C(x)$ need not have rank one. Distinct source faces may meet and their contributions may add to higher-rank tensors. Here the geometry of the overlaps created by the deformation enters only through the matrix $C(x)$. 

We now forget this geometric origin for a moment and derive a uniform estimate for every
$C\in M_3(\FF)$.

For $C\in M_3(\FF)$, define the real 0-1 matrix
\be\label{eq:BC}
B_C(g,h)=\mathbf1_{\{g^TCh=1\}},\qquad g,h\in\Gx.
\ee
Thus $C$ is a $3\times3$ $\FF$ matrix, while $B_C$ is a $7\times7$ real matrix, called the incidence matrix. 

\begin{lem}\label{lem:incidence}
Every row and every column of $B_C$ is either zero or contains exactly 4 ones. 
\end{lem}
\nd
For any fixed $g$, the row is the restriction of the linear functional $h\mapsto(C^Tg)\cdot h$. If $C^Tg=0$, the row vanishes. Otherwise its kernel in $G$ has four elements, and the affine coset on which it equals one also has four elements. The latter does not contain zero, so all 4 elements belong to $\Gx$.
\qed

See Figure~\ref{fig:incidence} for concrete examples.

\begin{figure}[H]\centering
\begin{tikzpicture}[x=.54cm,y=.54cm]

\begin{scope}
\foreach \r in {1,...,7}\foreach \s in {1,...,7}{
  \draw[gray!45] (\s-1,-\r+1) rectangle(\s,-\r);
  \node[gray!65,font=\scriptsize] at (\s-.5,-\r+.5){$0$};
}
\foreach \r in {2,3,4,5}\foreach \s in {2,3,4,5}{
  \fill[pale] (\s-1,-\r+1) rectangle(\s,-\r);
  \draw[gray!45] (\s-1,-\r+1) rectangle(\s,-\r);
  \node at (\s-.5,-\r+.5){$1$};
}
\foreach \i/\lab in {1/100,2/010,3/001,4/110,5/101,6/011,7/111}{
  \node[left,font=\scriptsize] at (-.12,-\i+.5){$\lab$};
  \node[above,rotate=45,anchor=west,font=\scriptsize] at (\i-.5,.08){$\lab$};
}
\node[left] at (-1.2,-3.5){$B_C=$};
\node[anchor=west,align=left,font=\small] at (9,-3.5){
(a) $\a=(0,1,1)^T$, $C=\a\a^T$. $\operatorname{rank} C$=1.\\[4pt]
Active row and column labels: $010,001,110,101$.\\[4pt]
Each nonzero row or column has four ones.\\[2pt]
The active positions form a $4\times4$ rectangle.};
\end{scope}

\begin{scope}[yshift=-5.7cm]
\foreach \r in {1,...,7}\foreach \s in {1,...,7}{
  \draw[gray!45] (\s-1,-\r+1) rectangle(\s,-\r);
  \node[gray!65,font=\scriptsize] at (\s-.5,-\r+.5){$0$};
}
\foreach \r/\s in {1/1,1/4,1/5,1/7,
                     2/2,2/4,2/6,2/7,
                     4/1,4/2,4/5,4/6,
                     5/1,5/4,5/5,5/7,
                     6/2,6/4,6/6,6/7,
                     7/1,7/2,7/5,7/6}{
  \fill[pale] (\s-1,-\r+1) rectangle(\s,-\r);
  \draw[gray!45] (\s-1,-\r+1) rectangle(\s,-\r);
  \node at (\s-.5,-\r+.5){$1$};
}
\foreach \i/\lab in {1/100,2/010,3/001,4/110,5/101,6/011,7/111}{
  \node[left,font=\scriptsize] at (-.12,-\i+.5){$\lab$};
  \node[above,rotate=45,anchor=west,font=\scriptsize] at (\i-.5,.08){$\lab$};
}
\node[left] at (-1.2,-3.5){$B_C=$};
\node[anchor=west,align=left,font=\small] at (9,-3.5){
(b) $C=\operatorname{diag}(1,1,0)$. $\operatorname{rank}C$=2.\\[4pt]
The row and column labelled $001$ are zero.\\[4pt]
Every other row or column has four ones.\\[2pt]
The active positions do not form a single rectangle.};
\end{scope}


\end{tikzpicture}
\caption{Incidence matrices of ranks one and two. In every case each row and each column contains either zero or four ones. At rank one the active positions form a single rectangular block after reordering the labels, whereas at rank two they do not.}
\label{fig:incidence}
\end{figure}

\subsection{Pointwise estimate in the tetrahedral case}

In this subsection we are going to control, at each point, the sum effect of all product area vectors $a_g\wedge a'_h$ such that the pair $(g,h)$ contribute to this point.

\begin{lem}\label{lem:signnorm}
If a real matrix $\mathsf S$ satisfies $|\mathsf S_{g,h}|=B_C(g,h)$ for every $g,h\in \Gx$, then
\be\label{eq:Sbound}
\|\mathsf S\|_{\op}\leq4.
\ee
Moreover, the bound is sharp: for the nonnegative incidence
matrix itself, $\|B_C\|_{\op}=4$ whenever $C\ne0$.
\end{lem}
\nd
For $z\in\R^7$, apply Cauchy--Schwarz row by row and then use Lemma~\ref{lem:incidence}, we obtain
\be
\begin{aligned}
\|\mathsf Sz\|_2^2
&=\sum_g\left|\sum_h\mathsf S_{g,h}z_h\right|^2
\leq\sum_g\left(4\sum_{h}(|\mathsf S_{g,h}||z_h|)^2\right)\\
&=4\sum_h|z_h|^2\left(\sum_gB_C(g,h)\right)
\leq16\|z\|_2^2.
\end{aligned}
\ee
Conversely, if $I,J$ are the nonzero row and column labels (i.e., $I$ is the 7-dimensional 0-1 vector so that its $i$-th coordinate is 1 if and only if the $i$-th row of $B_C$ is nonzero; the same for $J$) then $B_C\mathbf1_J=4\mathbf1_I$ and $|I|=|J|>0$. Hence the unsigned matrix $B_C$ attains norm four.
\qed

The incidence matrix tells us which pairs $(g,h)$ may contribute at a given point. We now need to control the size of the corresponding geometric vectors. This is where the symmetry  identity \eqref{eq:frame} enters.

Let $U,V$ be the real $3\times7$ matrices with columns $a_g,a'_h$, respectively. By \eqref{eq:frame},
\be\label{eq:UV}
UU^T=VV^T=\kappa I_3,\qquad
\|U\|_{\op}=\|V\|_{\op}=\sqrt\kappa\mbox{ with }\kappa=\frac{32}{9}.
\ee
The operator $U^T:n\mapsto(a_g\cdot n)_g$ with $n\in \R^3$ records the components of the seven area vectors in the direction $n$, and
$\|U^Tn\|_2^2=\kappa|n|^2$. Thus $U^T/\sqrt\kappa$ is an isometric embedding. Conversely,
for $z=(z_g)_{g\in\Gx}\in\R^7$, the matrix $U$ acts by
\[
Uz=\sum_{g\in\Gx}z_ga_g,
\qquad
|Uz|\leq\sqrt\kappa\,\|z\|_2.
\]
Now for $K\in M_3(\R)$, set
\be\label{eq:mixedform}
\omega_K=\sum_{i,j=1}^3K_{ij}\,dx_i\wedge dx'_j.
\ee
It is a two-form on $X\times X'\simeq \R^6$. Its comass is
\be
\|\omega_K\|_*=
\sup\{|\omega_K(u,v)|:|u|=|v|=1,\ u\perp v\}.
\ee
\begin{lem}\label{lem:comass}
$\|\omega_K\|_*=\|K\|_{\op}$.
\end{lem}
\nd
For $u=(x,x')$ and $v=(y,y')$,
\be
\omega_K(u,v)=x^TKy'-y^TKx'
=u^T\begin{pmatrix}0&K\\-K^T&0\end{pmatrix}v.
\ee
This block operator has norm $\|K\|_{\op}$, because its transpose times itself has diagonal blocks $KK^T$ and $K^TK$. Hence $\|\omega_K\|_*\leq\|K\|_{\op}$. Conversly, choose unit left and right singular vectors $n,m$ with $|n^TKm|=\|K\|_{\op}$. Evaluation of $\omega_K$ on the orthonormal pair $u=(n,0)$ and $v=(0,m)$ attains the bound.
\qed


We are now ready to combine the uniform incidence bound and the symmetry identity \eqref{eq:frame}.

Let $\mathcal U_2$ be the set of unit simple 2-vectors in $X\times X'$. For $C\in M_3(\FF)$ and $\eta\in\mathcal U_2$, set
\be\label{eq:Phi}
\Phi(C,\eta)=\sum_{g,h\in\Gx}\!B_C(g,h)
\left|\pair{\eta}{a_g\wedge a'_h}\right|.
\ee
\begin{thm}[Pointwise estimate]\label{thm:key}
For every $C\in M_3(\FF)$ and $\eta\in\mathcal U_2$,
\be\label{eq:key}
\Phi(C,\eta)\leq4\kappa=\frac{128}{9}.
\ee
The constant is optimal.
\end{thm}
\nd

Fix $\eta\in \mathcal U_2$. For each $g,h$ so that $B_C(g,h)\ne 0$, choose $\eps_{g,h}\in\{-1,1\}$ so that
\be
\eps_{g,h}\pair{\eta}{a_g\wedge a'_h}
=\left|\pair{\eta}{a_g\wedge a'_h}\right|.
\ee
When $\pair{\eta}{a_g\wedge a'_h}=0$, either sign may be chosen. Set $\mathsf S_{g,h}=\eps_{g,h}B_C(g,h)$ and define
\be
\omega=\sum_{g,h}\mathsf S_{g,h}a_g^\flat\wedge {a'_h}^\flat.
\ee
Then its coefficient matrix is $U\mathsf SV^T$. Lemma~\ref{lem:comass}, \eqref{eq:UV}, and Lemma~\ref{lem:signnorm} give
\begin{align}\label{eq:keychain}
\Phi(C,\eta)=\omega(\eta)
&\leq\|\omega\|_*=\|U\mathsf SV^T\|_{\op}\\
&\leq\|U\|_{\op}\|\mathsf S\|_{\op}\|V\|_{\op}
\leq\sqrt\kappa\,4\,\sqrt\kappa=4\kappa.\notag
\end{align}
To verify the sharpness, take $\a=011$, $C=\a\a^T$; choose $n=\frac1{\sqrt2}(0,1,-1)\in X, m=\frac1{\sqrt2}(0,1,-1)\in X'$, and set $\eta=(n,0)\wedge (0,m)$. Then the active row and column labels are $010,001,110,101$. Table~\ref{tab:vectors} gives, for every active $g$,
$|a_g\cdot n|=2\sqrt2/3$, and the same holds in the second factor. The sixteen terms sum to
$16(2\sqrt2/3)^2=128/9=4\kappa$.
\qed

\begin{figure}[H]\centering
\begin{tikzpicture}[>=Latex,node distance=2.45cm]
\node (p) {$\R^3_{X'}$};
\node[right=of p] (q) {$\R^7$};
\node[right=of q] (r) {$\R^7$};
\node[right=of r] (s) {$\R^3_X$};
\draw[->,thick] (p)--node[above]{$V^T$}node[below]{$\sqrt\kappa$}(q);
\draw[->,thick] (q)--node[above]{$\mathsf S$}node[below]{$4$}(r);
\draw[->,thick] (r)--node[above]{$U$}node[below]{$\sqrt\kappa$}(s);
\end{tikzpicture}
\caption{The factorization $U\mathsf S V^T$.
The map $V^T$ records the inner products with the vectors
$a'_h$, the signed incidence matrix $\mathsf S$ combines these
components, and $U$ forms the resulting linear combination of
the vectors $a_g$. The numbers below the arrows are
operator-norm bounds, giving
$\|U\mathsf S V^T\|_{\op}
\leq\sqrt\kappa\cdot4\cdot\sqrt\kappa=4\kappa$.}
\label{fig:operator}
\end{figure}

\begin{rem}\label{rem:triangle-product}As already mentioned in Remark \ref{rem:2.28}, now the reason for using the complete family becomes clearer in the product argument.

Let $\mathcal T=\{100,010,001,111\}\subset G^\times$
be the four triangular labels. Then on every original face $\Delta_{ij}$, we know that exactly two of the four triangular $g\in\mathcal T$ satisfy $g\cdot\a_{ij}=1$. In other words, each face $\Delta_{ij}$ activate two of the $D_g, g\in\mathcal T$. After a general noninjective deformation, different source faces may overlap and their $\FF$ multiplicities may add. In codimension one we can use separation to deal with this, so that the number of activation on $D_g, g\in\mathcal T$ at each point is  2. But when it comes to the product in codimension 2, the separation does not work, and for a ovelapped face, it might activate all four triangular cycles $D_g$ (e.g. $c=111$). The role of adding the three quadrilateral cycles is precisely to remove
this nonuniformity: 
every nonzero linear functional on $G$ takes the value one on
exactly four nonzero elements. In the product this becomes the
uniform four-point row and column count of $B_C$ in Lemma~\ref{lem:signnorm}
\end{rem}

Finally we build a bridge between the projection and the calibrations.
Set $\widehat a_g=a_g/A_g$ and $\widehat a'_h=a'_h/A'_h$, and define the projection
\be
\Pi_{g,h}=\rho_g\times\rho'_h:
X\times X'\longrightarrow a_g^\perp\times {a'_h}^\perp.
\ee
\begin{lem}\label{lem:jacobian}
If $L\subset X\times X'$ is a 4-dimensional linear plane and $\eta$ is a unit simple normal 2-vector to $L$, then
\be\label{eq:jacobian}
A_gA'_hJ_4(\Pi_{g,h}|_L)
=\left|\pair{\eta}{a_g\wedge a'_h}\right|.
\ee
\end{lem}
\nd
Choose a unit simple 4-vector $\xi$ spanning $L$ and a unit simple 4-vector $\zeta$ spanning the image plane $a_g^\perp\times {a'_h}^\perp$. The 4-dimensional Jacobian of the orthogonal projection is $|\pair\xi\zeta|$. The orthogonal complement of the image plane is spanned by the orthonormal vectors $\widehat a_g,\widehat a'_h$, so $*\zeta=\pm\widehat a_g\wedge\widehat a'_h$, while $*\xi=\pm\eta$. Since the Hodge star preserves inner products,
$J_4=|\pair\eta{\widehat a_g\wedge\widehat a'_h}|$. Multiply by $A_gA'_h$ and we get the conclusion.
\qed

\section{Almgren minimality of the tetrahedral product}\label{sec:minimality}

In this section we are going to prove the Almgren minimality of $\T\times \T'$.

\subsection{Minimality in the spanning competitor class}

Before proving the Almgren minimality for the set $E_0=S\times S'$, we will first work with a larger spanning class. The reason is that the projection argument below uses only the boundary currents; it does not require the competitor currents to come from one common deformation. Once the comparison is proved in this larger class, we will return to an arbitrary Almgren deformation and verify that it belongs to the class.

 Let us take all the notations as before.

\begin{defn}[The spanning competitor class]\label{def:currentclass}
Let $\mathscr C(\Dcal)$ be the class of all the compact sets $F\subset\Dcal$ such that $F$ is 4-rectifiable, $\H^4(F)<\infty$, and there are nine currents $\Gamma_{rs}\in\I_4(\R^6;\FF)$ satisfying
\be\label{eq:prescribedboundary}
\spt\Gamma_{rs}\subset F,\qquad
\partial\Gamma_{rs}=\partial(D_{e_r}\times D'_{e_s})
\quad(1\leq r,s\leq3).
\ee
\end{defn}

\begin{lem}\label{lem:fixedprojection}
Let $F\in\mathscr C(\Dcal)$, choose the nine currents $\Gamma_{rs}, 1\le r,s\le 3$ in its definition, and define $\Gamma_{g,h}$ by \eqref{eq:49}. Then for every $g,h\in\Gx$, 
\be\label{eq:degreeone}
(\Pi_{g,h})_{\s}\Gamma_{g,h}=\chain{Q_g\times Q'_h}.
\ee
\end{lem}
\nd
By \eqref{eq:prescribedboundary}, we have
\begin{align*}
\partial\Gamma_{g,h}
&=\pa\left(\sum_{r,s}g_rh_s\Gamma_{rs}\right)=\sum_{r,s}g_rh_s\partial(\Gamma_{rs})
=\sum_{r,s}g_rh_s\partial(D_{e_r}\times D'_{e_s})\\
&=\partial\left(\sum_{r,s}g_rh_sD_{e_r}\times D'_{e_s}\right)
=\partial(D_g\times D'_h).
\end{align*}
On the other hand, Lemma~\ref{lem:polygon} implies
\be
(\Pi_{g,h})_{\s}(D_g\times D'_h)
=((\rho_g)_{\s}D_g)\times((\rho'_h)_{\s}D'_h)
=\chain{Q_g\times Q'_h}.
\ee
Consequently if we set
\be
A=(\Pi_{g,h})_{\s}\Gamma_{g,h}+\chain{Q_g\times Q'_h},
\ee
then $A$ is a compactly supported $\FF$-integral four-current in $a_g^\perp\times {a'_h}^\perp$, so that
\be
\partial A=(\Pi_{g,h})_{\s}\bigl(\partial\Gamma_{g,h}
+\partial(D_g\times D'_h)\bigr)=0.
\ee
As a result, the constancy lemma~\ref{lem:topcycle} gives $A=0$, which gives the conclusion.
\qed

The above lemma tells that  every test current keeps the same mod-two projection as the corresponding current for the reference set $E_0$. This gives the lower projection bounds needed for the area comparison:

\begin{cor}\label{cor:projectionlower}
Under the assumptions of the preceding lemma,
\be\label{eq:projectionlower}
A_gA'_h\leq\int_{\car\Gamma_{g,h}}
J_4(\Pi_{g,h}|_{T_xF})\,d\H^4(x)
\quad(g,h\in\Gx).
\ee
\end{cor}
\nd Apply Lemma \ref{lem:pushmass} then Lemma \ref{lem:fixedprojection}, we have
\be\begin{split} \int_{\car\Gamma_{g,h}}J_4(\Pi_{g,h}|_{T_xF})\,d\H^4(x)&\ge \mass((\Pi_{g,h})_{\s}\Gamma_{g,h})=\H^4(Q_g\times Q'_h)\\
&=\H^2(Q_g)\H^2(Q'_h)=A_gA'_h.
\end{split}\ee

\qed

We have now obtained the two estimates needed for the comparison. The fixed boundary gives the above lower bound for every projected current, while Theorem~\ref{thm:key} gives a uniform upper bound for the sum of the corresponding Jacobians. 

\begin{thm}[Minimality in the spanning competitor class]\label{thm:currentminimal}
For every $F\in\mathscr C(\Dcal)$,
\be\label{eq:areabound}
\H^4(F)\geq8=\H^4(E_0).
\ee
\end{thm}
\nd Fix any $F\in \mathscr C(\Dcal)$, choose the $\Gamma_{rs},1\le r,s\le 3$ as in the definition of $\mathscr C(\Dcal)$. Let $C(x)$ be the multiplicity matrix for $\Gamma_{rs}, 1\le r,s\le 3$ on the common carrier $F$, as defined in \eqref{eq:multiplicityC}. Let $\eta_x$ be a measurable unit simple 2-vector fields on $F$, so that $\eta_x\perp T_xF$ almost everywhere.

Multiply \eqref{eq:projectionlower} by the positive weight $A_gA'_h$ and sum over all $g,h$. Then Proposition~\ref{prop:parity}, Lemma~\ref{lem:jacobian}, and Theorem~\ref{thm:key} yield
\be\begin{split}\label{eq:comparisonchain}
&\left(\sum_{g\in\Gx}A_g^2\right)
\left(\sum_{h\in\Gx}{A'_h}^2\right)
=\sum_{g,h\in\Gx}A_g^2{A'_h}^2\\
&\leq\sum_{g,h\in\Gx}A_gA'_h
\int_F\mathbf1_{\{g^TC(x)h=1\}}
J_4(\Pi_{g,h}|_{T_xF})\,d\H^4(x)\\
&=\int_F\Phi(C(x),\eta_x)\,d\H^4(x)\leq4\kappa\,\H^4(F).
\end{split}\ee
By \eqref{eq:frame}, the left-hand side equals $(3\kappa)^2$. Dividing both sides by $4\kappa$ gives $\H^4(F)\geq9\kappa/4=8$.

On the other hand, the thirty-six four-dimensional faces $\Delta_{ij}\times \Delta'_{kl}, i<j,k<l$ of $E_0$ are products of orthogonal triangles. Their relative interiors are disjoint, and their intersections have dimension at most three. Hence
\be\label{eq:referencearea}
\H^4(E_0)=\sum_{i<j,\,k<l}\H^4(\Delta_{ij}\times\Delta'_{kl})
=36\left(\frac{\sqrt2}{3}\right)^2=8.
\ee
The currents $D_{e_r}\times D'_{e_s}$ show that $E_0\in\mathscr C(\Dcal)$, so the lower bound is attained.
\qed

\subsection{Almgren minimality of $E=\T\times\T'$}

We now come to the Almgren minimality of $\T\times \T'$. It remains to pass from the spanning competitor class back to Almgren competitors. For an arbitrary Almgren deformation, we push forward the nine basic currents and verify that their boundaries are unchanged.

\noindent\textbf{Proof of Theorem~\ref{main:min}. }

Take all the notations as before. Since $E=\T\times \T'$ is a cone, it suffices to prove its Almgren minimality in the domain $\mathring\Dcal=\mathring\Delta\times \mathring\Delta'$.

Let $\{\varphi_t\}$ be an arbitrary Almgren deformation in $\mathring\Dcal$, and set $f=\varphi_1$. Set
\be\label{eq:Fequal}
F:=f(E)\cap\Dcal=f(E\cap\Dcal)=f(E_0).
\ee
This set is a finite union of Lipschitz images of compact four-dimensional polyhedra. It is therefore compact, 4-rectifiable, and has finite $\H^4$-measure.

Define $\Gamma_{rs}=f_{\s}(D_{e_r}\times D'_{e_s})$. These currents are supported in $F$. By \eqref{eq:productboundary}, the boundaries of the source currents are supported on $\partial\Dcal$, where $f$ is the identity. Thus
\be
\partial\Gamma_{rs}
=f_{\s}\partial(D_{e_r}\times D'_{e_s})
=\partial(D_{e_r}\times D'_{e_s}).
\ee
Consequently $F\in\mathscr C(\Dcal)$, and Theorem~\ref{thm:currentminimal} gives
\be
\H^4(f(E)\cap\Dcal)=\H^4(F)\geq8=\H^4(E\cap\Dcal).
\ee
Subtracting the common finite term
$\H^4(E\cap f(E)\cap\Dcal)$ from both sides, we obtain
\be
\H^4(E\bs f(E))\leq\H^4(f(E)\bs E).\ee\qed

\section{Almgren minimality of general simplex-skeleton products}\label{sec:simplex-products}

In this section we will extend the preceding argument from the tetrahedral cone to the whole regular-simplex family. We first construct the complete $\FF$ cycle family in one factor, introduce the associated projection and volume (in place of area) vectors, and prove their symmetry identity. Then we will pass to products and generalize the incidence-matrix, the pointwise estimate, and the spanning competitor arguments.

We therefore emphasize only the points where the dimension enters and where the notation differs from the tetrahedral case.

Let us first state the theorem that we are going to prove in this section.

\begin{thm}[Minimality of the orthogonal products of regular simplex-skeleton cones]\label{thm:simplex-main}
Let $m,n\geq2$, and let $X\simeq\R^m$, $X'\simeq\R^n$ be orthogonal. In each factor choose the cone over the codimension-two skeleton of a regular simplex centered at the origin, denoted by $S_m\subset X$ and $S_n\subset X'$. Then
\be\label{eq:sp-main}
E_{m,n}=S_m\times S_n\subset X\times X'
\ee
is an Almgren minimal cone of dimension $d=m+n-2$.
\end{thm}

Here each factor has codimension one, hence the orthogonal product always has codimension two. 

Theorem \ref{thm:simplex-main} is the detailed form of Theorem~\ref{main:simplex}. As in the tetrahedral case, competitors need not be injective, preserve the product structure, or retain the original number of faces. 

We will first discuss notations and geometric properties in a single $\R^k$ in Subsections~\ref{subsec:simplex-model}--\ref{subsec:simplex-frame}.
We will only turn to product spaces from
Subsection~\ref{subsec:simplex-incidence}.

\subsection{The simplex-skeleton cone}\label{subsec:simplex-model}
Let  $k\geq2$ be an integer. In $\R^k$ we choose a regular simplex centered at 0:
\be\label{eq:sp-simplex}
\Delta^k=\conv\{v_0,\ldots,v_k\},\qquad
|v_i|=1,\qquad v_i\cdot v_j=-\frac1k\quad(i\ne j).
\ee
The vertices span $\R^k$ and satisfy $\sum_{i=0}^k v_i=0$.

Let $F_i$ be the $(k-1)$-face opposite $v_i$, and let $\operatorname{skel}_{k-2}(\Delta^k)$ denote the union of all faces of dimension at most $k-2$. Set
\be\label{eq:sp-cone}
\begin{split}
S_k&=\{t x:t\geq0,\ x\in\operatorname{skel}_{k-2}(\Delta^k)\},\\
K_{ij}&=\conv\bigl(\{0\}\cup(F_i\cap F_j)\bigr)
\qquad(0\leq i<j\leq k).
\end{split}
\ee
So $S_k$ is the cone over the $k-2$-skeleton of $\Delta^k$, and $K_{ij}, 0\le i<j\le k$ are the $k-1$-dimensional faces of $S_k\cap \Delta^k$. Each $K_{ij}$ is a $(k-1)$-simplex, and
\be\label{eq:sp-sectors}
S_k\cap\Delta^k=\bigcup_{0\leq i<j\leq k}K_{ij}.
\ee
The above union is essentially disjoint, in the sense that the $K_{ij}$ have disjoint relative interiors and intersections of dimension at most $k-2$.

For $k=2$, $S_2$ is the 1-dimensional $\Y$-set consisting of three rays at $120^\circ$; for $k=3$, $S_3=\T$. Here $K_{ij}$ is indexed by \emph{adjacent outer facets}. In the tetrahedral case, if $\{i,j,l,r\}=\{0,1,2,3\}$, then $K_{ij}=\Delta_{lr}$.

\bigskip

We write
\be\label{eq:sp-facet-data}
w_i=-v_i,\qquad q_k=\H^{k-1}(F_i) \mbox{ (independent of $i$) }.
\ee
Then $w_i$ is the outward unit normal of $F_i$. We also have
\be\label{eq:sp-wframe}
\sum_{i=0}^k w_i=0,\qquad
w_i\cdot w_j=-\frac1k\ (i\ne j),\qquad
\sum_{i=0}^k w_iw_i^T=\frac{k+1}{k}I_k.
\ee

\subsection{The complete $\FF$ cycle family}\label{subsec:simplex-currents}
Let $G_k=\FF^k$ be the coefficient space, and let $e_i, 1\le i\le k$ be the standard basis of $G_k=\FF^k$.

Set
\be\label{eq:sp-labels}
c_0=0,\qquad c_i=e_i\ (1\leq i\leq k),\qquad
\d_{ij}=c_i+c_j\in G_k.
\ee
Write $G_k^\times=G_k\bs\{0\}$.

\begin{rem}Here the notation $\d_{ij}$ should be distinguished from the
tetrahedral label $\a_{ij}$ introduced in \eqref{eq:alpha}.
The difference comes only from the indexing convention. In the
tetrahedral case, $\Delta_{ij}$ is indexed by its two vertices,
while here $K_{ij}$ is indexed by the two facets $F_i,F_j$.
If $\{i,j,p,q\}=\{0,1,2,3\}$, then $K_{ij}=\Delta_{pq},\qquad \d_{ij}=\a_{pq}$. Thus the two notations encode the same multiplicity rule.
\end{rem}

For each $g=(g_1,\ldots,g_k)\in G_k$, define the $\FF$ polyhedral currents
\be\label{eq:sp-current-chains}
\begin{aligned}
P_g&=\sum_{i=1}^k g_i\chain{F_i}, &\lambda_g&=\partial P_g,\\
U_g&=0*P_g, &D_g&=P_g+\partial U_g.
\end{aligned}
\ee
The currents $P_g,D_g$ have dimension $k-1$, $\lambda_g$ has dimension $k-2$, and $U_g$ has dimension $k$. All four currents vanish for $g=0$.

\begin{lem}\label{lem:sp-current}
For every $g\in G_k$,
\be\label{eq:sp-Dformula}
D_g=\sum_{0\leq i<j\leq k}
\bigl(g\cdot\d_{ij}\bigr)\chain{K_{ij}},\qquad
\partial D_g=\lambda_g.
\ee
In particular, $\spt D_g\subset S_k\cap\Delta^k$, $\spt\partial D_g\subset\partial\Delta^k$, and $g\mapsto D_g$ is $\FF$-linear.
\end{lem}
\nd
For $0\leq i\leq k$, set $\mathcal P_i=\conv(\{0\}\cup F_i)$. Its boundary consists of $F_i$ and the $k$ radial faces $K_{ij}$, $j\ne i$, so
\be
\partial\chain{\mathcal P_i}
=\chain{F_i}+\sum_{j\ne i}\chain{K_{ij}}.
\ee
Here we temporarily set $K_{ji}=K_{ij}$. Also $U_g=\sum_{i=1}^k g_i\chain{\mathcal P_i}$. Hence by \eqref{eq:sp-current-chains},
\be D_g=P_g+\pa U_g=\sum_{i=1}^k g_i\sum_{j\ne i}\chain{K_{ij}}.\ee
Thus the coefficient of the radial face $K_{ij}$ is $(g\cdot c_i)+(g\cdot c_j)=g\cdot\d_{ij}$.
This proves the first identity. The second follows from
$\partial D_g=\partial P_g+\partial^2U_g=\partial P_g$. All constructions are linear modulo two.
\qed

Since
\[
D_g=\sum_{r=1}^k g_rD_{e_r},
\qquad g\in G_k,
\]
the family $\{D_g:g\in G_k^\times\}$ consists precisely of all nonzero $\FF$ combinations of the $k$ basic currents
$D_{e_1},\ldots,D_{e_k}$. Thus the complete family contains $2^k-1$ currents. On the face $K_{ij}$, the multiplicity of
$D_g$ is $g\cdot\d_{ij}$.

\subsection{Projection and volume vectors}\label{subsec:simplex-projections}
For $g\in G_k^\times$, define
\be\label{eq:sp-area-vector}
a_g=q_k\sum_{i=1}^k g_iw_i\in\R^k.
\ee
Unlike the current map $g\mapsto D_g$, the map $g\mapsto a_g$ is not regarded as $\FF$-linear. Set
\be
I_g=\{i\in\{1,\ldots,k\}:g_i=1\},\qquad s=|I_g|.
\ee
Since $g\ne0$, we know that $1\leq s\leq k$. By \eqref{eq:sp-wframe}, we have
\be\label{eq:sp-signs}
w_i\cdot a_g=
\begin{cases}
q_k(k+1-s)/k,&i\in I_g,\\
-q_ks/k,&i\notin I_g,
\end{cases}
\qquad(0\leq i\leq k).
\ee
As a result,
\be\label{eq:sp-vector-length}
\begin{split}
|a_g|^2&=q_k^2\left|\sum_{i\in I_g}w_i\right|^2=q_k^2\left(\sum_{i\in I_g}|w_i|^2+2\sum_{i<j,i,j\in I_g}w_i\cdot w_j\right)\\
&=q_k^2\left(s+2\binom{s}{2}\left(-\frac1k\right)
\right)=q_k^2\left(s-\frac{s(s-1)}k\right)\\
&=q_k^2\frac{ks-s^2+s}{k}=q_k^2\frac{s(k+1-s)}{k}>0.
\end{split}\ee
Therefore we can define
\be\label{eq:sp-projectors}
\rho_g:\R^k\longrightarrow a_g^\perp,
\qquad Q_g=\rho_g(\Delta^k),\qquad A_g=|a_g|,
\ee
where $\rho_g$ is an orthogonal projection.

\begin{lem}\label{lem:sp-projection}
The set $Q_g$ is a convex polytope in $a_g^\perp$, with the origin in its relative interior, and
\be\label{eq:sp-projection-current}
(\rho_g)_{\s}D_g=\chain{Q_g},\qquad
\H^{k-1}(Q_g)=A_g.
\ee
\end{lem}
\nd
The convexity and the polytope property follow directly from linear projection. Since $0\in\operatorname{int}\Delta^k$, a small ball centered at the origin lies in $\Delta^k$. Its projection is a ball in $a_g^\perp$, so the origin is a relative interior point of $Q_g$.

Set $\widehat a_g=a_g/A_g$. By \eqref{eq:sp-signs}, the facets $F_i,i\in I_g$ selected in $P_g$ are exactly the facets whose outward normals have positive scalar product with $\widehat a_g$, and the other facets $F_i,i\not\in I_g$ have negative scalar product. Therefore, the facets $F_i,i\in I_g$ are exactly the upper facets in direction $\widehat a_g$. 

Now we index the line parallel to $a_g$ that meets $\Delta^k$ by $z+\R\widehat a_g, z\in Q_g$. Then by convexity, for $z\in Q_g$, $z+\R\widehat a_g$ intersects $\Delta^k$ in a closed interval, and is nondegenerate if $z\not\in \pa Q_g$. Moreover, if $z\not\in\rho_g(\operatorname{skel}_{k-2}(\Delta^k))$, the upper endpoint lies in the relative interior of exactly one of  the facets $F_i,i\in I_g$. Note that the set $\rho_g(\operatorname{skel}_{k-2}(\Delta^k))$ is a finite union of polytopes of dimension at most $k-2$, hence is $\H^{k-1}$-null; $\pa Q_g$ is also $\H^{k-1}$-null.

As a result, the selected facets $F_i,i\in I_g$ of $P_g$ project onto $Q_g$ with exactly one covering almost everywhere and no contribution outside $Q_g$. This proves multiplicity one, not only coverage of the support:
\be\label{eq:sp-upperprojection}
(\rho_g)_{\s}P_g=\chain{Q_g}.
\ee
On the other hand, $U_g$ has dimension $k$, while $\rho_g$ has rank $k-1$. The area formula gives $(\rho_g)_{\s}U_g=0$. As a result, by \eqref{eq:sp-current-chains},
\be
(\rho_g)_{\s}D_g
=(\rho_g)_{\s}P_g+\partial((\rho_g)_{\s}U_g)
=(\rho_g)_{\s}P_g.
\ee
Together with \eqref{eq:sp-upperprojection}, we get the first identity of \eqref{eq:sp-projection-current}.

Finally, on each selected facet $F_i$, $i\in I_g$, the $(k-1)$-dimensional projection Jacobian is
$|w_i\cdot\widehat a_g|=w_i\cdot\widehat a_g$. Then by the almost-everywhere single coverage just proved, the area formula gives
\be
\H^{k-1}(Q_g)=\sum_{i\in I_g}\int_{F_i}w_i\cdot\widehat a_g d\H^{k-1}
=q_k\sum_{i\in I_g}w_i\cdot\widehat a_g
=a_g\cdot\widehat a_g=A_g.
\ee
\qed

\subsection{Symmetry of the volume vectors and the measure of $S_k\cap \Delta^k$}\label{subsec:simplex-frame}
Now we give the general-dimensional analogue of Lemma~\ref{lem:frame}. It gives the geometric $L^2$ factor in
the product estimate below.
\begin{lem}[Symmetry identity in arbitrary dimension]\label{lem:sp-frame}
Define
\be\label{eq:sp-alpha}
\a_k=q_k^2\,2^{k-2}\frac{k+1}{k}.
\ee
Then

$1^\circ$ 
\be\label{eq:sp-frame}
\sum_{g\in G_k^\times}a_ga_g^T=\a_k I_k,
\qquad
\sum_{g\in G_k^\times}A_g^2=k\a_k.
\ee

$2^\circ$ Let $\mathsf A_k$ be the matrix so that the columns of $\mathsf A_k$ are all the $a_g$, then
$\mathsf A_k\mathsf A_k^T=\a_kI_k$. In particular, for every $z\in\R^k$,
\be\label{eq:sp-l2}
\sum_{g\in G_k^\times}|a_g\cdot z|^2=\a_k|z|^2.
\ee
Consequently the operator
$z\mapsto(a_g\cdot z)_g=\mathsf A_k^Tz$ has $L^2$ amplification exactly $\sqrt{\a_k}$, and
$\|\mathsf A_k\|_{\op}=\sqrt{\a_k}$.

\end{lem}
\nd $1^\circ$
We temporarily regard each $g$ as a real 0-1 column vector. Then among the $2^k$ such vectors $g\in G_k$, a given coordinate equals one $2^{k-1}$ times, and two distinct coordinates are simultaneously one $2^{k-2}$ times. The zero vector contributes nothing. Hence
\be\label{eq:sp-bit-count}
\sum_{g\in G_k^\times}gg^T=2^{k-1}I_k+2^{k-2}
\begin{pmatrix}
0 & 1 & 1 & \cdots & 1\\
1 & 0 & 1 & \cdots & 1\\
1 & 1 & 0 & \cdots & 1\\
\vdots & \vdots & \vdots & \ddots & \vdots\\
1 & 1 & 1 & \cdots & 0
\end{pmatrix}
=2^{k-2}(I_k+\mathbf 1\mathbf1^T),
\ee
where $\mathbf1$ is the vector in $\R^k$ so that every coordinate equals 1.

Let $W$ be the real $k\times k$ matrix with columns $w_1,\ldots,w_k$. Then $a_g=q_kWg$.
Since $\sum_{i=1}^k w_i=-w_0$, we have, by \eqref{eq:sp-wframe},
\begin{align*}
\sum_{g\in G_k^\times}a_ga_g^T
&=q_k^2W\left(\sum_{g\in G_k^\times}gg^T\right)W^T
=q_k^2\,2^{k-2}
W(I_k+\mathbf1\mathbf1^T)W^T\\
&=q_k^2\,2^{k-2}\left[\sum_{i=1}^k w_iw_i^T+\left(\sum_{i=1}^k w_i\right)
\left(\sum_{i=1}^k w_i\right)^T\right]\\
&=q_k^2\,2^{k-2}[\frac{k+1}{k}I_k-w_0w_0^T+w_0w_0^T]\\
&=q_k^2\,2^{k-2}\frac{k+1}{k}I_k=\a_k I_k.
\end{align*}
This proves the first identity of \eqref{eq:sp-frame}. Then take the traces and note that $A_g=|a_g|$, we get directly the  second identity of \eqref{eq:sp-frame}.

$2^\circ$ Directly from $1^\circ$ we have $\mathsf A_k\mathsf A_k^T=\a_kI_k$. Note that for every $z\in \R^k$, $a_g\cdot z=z^Ta_g$, thus $z^T\mathsf A_k=(z^Ta_g)_{g\in G_k^\times}$, and therefore 
\be\begin{split} \sum_{g\in G_k^\times}|a_g\cdot z|^2&=z^T\mathsf A_k(z^T\mathsf A_k)^T=(z^Ta_g)_{g\in G_k^\times}[(z^Ta_g)_{g\in G_k^\times}]^T\\
&=z^T\mathsf A_k\mathsf A_k^Tz=z^T\a_kI_k z=\a_k|z|^2
\end{split}\ee
\qed

\begin{lem}\label{lem:sp-model-volume}
Let $M_k=\H^{k-1}(S_k\cap\Delta^k)$. Then
\be\label{eq:sp-q}
q_k=\H^{k-1}(F_i)=\frac{\sqrt{k}}{(k-1)!}
\left(\frac{k+1}{k}\right)^{(k-1)/2},
\ee
\be\label{eq:sp-sector-volume}
\H^{k-1}(K_{ij})
=\frac1{(k-1)!}
\sqrt{\frac2k\left(\frac{k+1}{k}\right)^{k-2}},
\ee
and
\be\label{eq:sp-model-volume}
M_k=q_k\sqrt{\frac{k(k+1)}2}
=\frac{k\sqrt{\a_k}}{2^{(k-1)/2}}.
\ee
\end{lem}
\nd
A facet $F_i$ has $k$ vertices. Fix one as a base point. Then the length of the edge vectors to the other $k-1$ vertices is $\sqrt\frac{2(k+1)}{k}$, and the angles between these vectors are $\pi/3$. Hence the Gram matrix of the edge vectors to the other $k-1$ vertices is
\be
\frac{k+1}{k}
\bigl(I_{k-1}+\mathbf1\mathbf1^T\bigr).
\ee
Its determinant is $k((k+1)/k)^{k-1}$, because
$I_{k-1}+\mathbf1\mathbf1^T$ has eigenvalue $k$ once and eigenvalue one with multiplicity $k-2$. Since the simplex volume is the square root of the Gram determinant divided by $(k-1)!$, we have proved \eqref{eq:sp-q}.

The nonzero vertices of $K_{ij}$ are the $k-1$ vectors $v_l$ with $l\notin\{i,j\}$. With the origin as base point, by \eqref{eq:sp-simplex}, their Gram matrix is
\be
\frac{k+1}{k}I_{k-1}-\frac1k\mathbf1\mathbf1^T.
\ee
It has eigenvalue $2/k$ on the all-ones direction and eigenvalue $(k+1)/k$ on its orthogonal complement, of dimension $k-2$. Then the same volume formula gives \eqref{eq:sp-sector-volume}. 

By \eqref{eq:sp-sectors}, there are $\binom{k+1}{2}$ sectors $K_{ij}$ of equal volume and pairwise $\H^{k-1}$-null intersections. Thus
\be
M_k=\binom{k+1}{2}\H^{k-1}(K_{01}).
\ee
Substituting the first two formulas and canceling their common factors gives
$M_k=q_k\sqrt{k(k+1)/2}$. Substitution of \eqref{eq:sp-alpha} gives the second identity of \eqref{eq:sp-model-volume}.
\qed

\subsection{Product coefficients and incidence matrices}\label{subsec:simplex-incidence}

After the above preparation in one factor, now we turn to product spaces. 

We apply the preceding construction separately in $X\simeq\R^m$ and $X'\simeq\R^n$. Use the notation $D_g,a_g,A_g,Q_g,\rho_g$ in the first factor and
$D'_h,a'_h,A'_h,Q'_h,\rho'_h$ in the second, with $g\in G_m^\times$, $h\in G_n^\times$. Set
\be\label{eq:sp-product-domain}
\mathcal D_{m,n}=\Delta^m\times\Delta^n,\qquad
E^0_{m,n}=E_{m,n}\cap\mathcal D_{m,n},\qquad d=m+n-2.
\ee
All products in this section are taken in $X\times X'$.

The coefficient space $G_m\otimes_\FF G_n$ has dimension $mn$, with basis $e_r\otimes e_s$. The associated basic geometric currents are
$D_{e_r}\times D'_{e_s}$, of dimension $(m-1)+(n-1)=d$. Then
\be\label{eq:sp-product-linear}
D_g\times D'_h
=\sum_{r=1}^m\sum_{s=1}^n g_rh_s
\bigl(D_{e_r}\times D'_{e_s}\bigr).
\ee
Here $mn$ is the dimension of the coefficient space, not the geometric dimension $d$ of the set.

Now let $F\subset X\times X'$ be a $d$-rectifiable set with finite measure, and let $\Gamma_{rs},1\le r\le m,1\le s\le n$ be $mn$ compactly supported $\FF$-integral currents whose essential carriers lie in $F$.  

As in the tetrahedral case, the local overlap pattern will be encoded by a $\FF$ multiplicity matrix
\be\label{eq:sp-C}
C(x)=(\theta_{rs}(x))\in M_{m,n}(\FF).
\ee
Note that the matrix $C(x)$ is no longer a square one when $m\ne n$.

In the chosen basis, $C(x)$ may also be regarded as the coefficient tensor $\sum_{r,s}\theta_{rs}(x)e_r\otimes e_s\in G_m\otimes_\FF G_n$. The current
$\Gamma_{g,h}=\sum_{r,s}g_rh_s\Gamma_{rs}$ has multiplicity
\be\label{eq:sp-parity}
\theta_{g,h}(x)=\sum_{r,s}g_rh_s\theta_{rs}(x)
=g^TC(x)h\quad\hbox{in }\FF.
\ee
Thus at a point $x$, the current $\Gamma_{g,h}$ is present with multiplicity one exactly when this sum equals 1 in $\FF$. Overlaps of different preimages, folding, and even cancellation are already reflected in the $\theta_{rs}$; no injectivity hypothesis is needed.

For any $C\in M_{m,n}(\FF)$, define the real incidence matrix
\be\label{eq:sp-incidence}
\mathsf B_C(g,h)=\mathbf1_{\{g^TCh=1\}},
\qquad g\in G_m^\times,\ h\in G_n^\times.
\ee
Its size is $(2^m-1)\times(2^n-1)$. 

\begin{lem}\label{lem:sp-incidence}

$1^\circ$
Every nonzero row of $\mathsf B_C$ contains $2^{n-1}$ ones, and every nonzero column contains $2^{m-1}$ ones.

$2^\circ$ If a real matrix $\mathsf S$ satisfies $|\mathsf S_{g,h}|=\mathsf B_C(g,h)$ for every $g\in G_m^\times$ and $h\in G_n^\times$, then
\be\label{eq:sp-signnorm}
\|\mathsf S\|_{\op}\leq2^{(m+n-2)/2}.
\ee
\end{lem}
\nd $1^\circ$
For each fixed $g\ne0$, the map $h\mapsto g^TCh$ is a linear functional on $\FF^n$. It either vanishes identically or has a kernel of size $2^{n-1}$, with the same number of elements in its coset of value one. Zero is not in that coset, so all of its $2^{n-1}$ elements belong to $G_n^\times$. 
Similar for the column count.

$2^\circ$ For any real column vector $z=(z_h)$, we apply Cauchy--Schwarz row by row and get
\be\begin{split}
\|\mathsf Sz\|_2^2
&=\sum_g\left|\sum_{h:\mathsf B_C(g,h)=1}
\mathsf S_{g,h}z_h\right|^2\leq 2^{n-1}\sum_g\sum_{h:\mathsf B_C(g,h)=1}|z_h|^2\\
&=2^{n-1}\sum_h |z_h|^2\sum_{g:\mathsf B_C(g,h)=1}
= 2^{n-1}2^{m-1}\sum_h|z_h|^2=2^{m+n-2}|z|^2.
\end{split}\ee
This gives \eqref{eq:sp-signnorm}. \qed

\subsection{Pointwise estimate in general dimensions}\label{subsec:simplex-estimate}
Let $\mathcal U_2(X\times X')$ be the set of unit simple 2-vectors in $X\times X'$. For $C\in M_{m,n}(\FF)$ and $\eta\in\mathcal U_2(X\times X')$, define
\be\label{eq:sp-Phi}
\Phi_{m,n}(C,\eta)
=\sum_{\substack{g\in G_m^\times,\ h\in G_n^\times\\g^TCh=1}}
\left|\pair{\eta}{a_g\wedge a'_h}\right|.
\ee
We first restate several linear-algebra facts of Section~\ref{sec:estimate} in general dimensions.

\begin{lem}\label{lem:sp-mixed}

$1^\circ$
For a real $m\times n$ matrix $K$, the two-form
$\omega_K=\sum_{i=1}^m\sum_{j=1}^nK_{ij}\,dx_i\wedge dx'_j$ satisfies
\be\label{eq:sp-comass}
\|\omega_K\|_*=\|K\|_{\op}.
\ee

$2^\circ$
If $L\subset X\times X'$ is a $d$-plane, $\eta$ is a unit simple normal 2-vector to $L$, and
$\Pi_{g,h}=\rho_g\times\rho'_h$, then
\be\label{eq:sp-jacobian}
A_gA'_hJ_d(\Pi_{g,h}|_L)
=\left|\pair{\eta}{a_g\wedge a'_h}\right|.
\ee
\end{lem}
\nd $1^\circ$
As in Lemma~\ref{lem:comass}, for $u=(x,x')$ and $v=(y,y')$,
\be
\omega_K(u,v)=x^TKy'-y^TKx'
=u^T\begin{pmatrix}0&K\\-K^T&0\end{pmatrix}v.
\ee
The transpose of the above block matrix times the matrix itself is $\operatorname{diag}(KK^T,K^TK)$, so $\|\omega_K\|_*\leq\|K\|_{\op}$. Conversely choose unit left and right singular vectors $x\in X,x'\in X'$. Then the orthonormal pair $(x,0),(0,x')$ attains the equality.

$2^\circ$ The target $d$-plane of $\Pi_{g,h}$ is $a_g^\perp\times {a'_h}^\perp$, with orthogonal complement spanned by the orthonormal vectors $a_g/A_g,a'_h/A'_h$. As in the proof of Lemma \ref{lem:jacobian}, we apply the orthogonal projection Jacobian formula to unit simple $d$-vectors of the source and target, and get
\be
J_d(\Pi_{g,h}|_L)
=\left|\pair{\eta}{(a_g/A_g)\wedge(a'_h/A'_h)}\right|.
\ee
Multiply by $A_gA'_h$ we get \eqref{eq:sp-jacobian}.
\qed

\begin{thm}[Pointwise product estimate]\label{thm:sp-pointwise}
For every $C\in M_{m,n}(\FF)$ and $\eta\in\mathcal U_2(X\times X')$,
\be\label{eq:sp-key}
\Phi_{m,n}(C,\eta)\leq\Lambda_{m,n}
\ee
where $\Lambda_{m,n}=2^{(m+n-2)/2}\sqrt{\a_m\a_n}$.
The constant is optimal.
\end{thm}
\nd Fix any $\eta$. For each $g,h$, choose signs $\eps_{g,h}\in \{-1,1\}$ so that
$\eps_{g,h}\pair{\eta}{a_g\wedge a'_h}
=|\pair{\eta}{a_g\wedge a'_h}|$, and set
$\mathsf S_{g,h}=\eps_{g,h}\mathsf B_C(g,h)$. Let $\mathsf A,\mathsf B$ be matrices with columns $a_g,a'_h$, respectively, and thus they are of sizes
$m\times(2^m-1)$ and $n\times(2^n-1)$.
Lemma~\ref{lem:sp-frame} gives
\be
\mathsf A\mathsf A^T=\a_mI_m,\qquad
\mathsf B\mathsf B^T=\a_nI_n.
\ee
The coefficient matrix of $\omega=\sum_{g,h}\mathsf S_{g,h}a_g^\flat\wedge {a'_h}^\flat$ is $\mathsf A\mathsf S\mathsf B^T$, so Lemmas~\ref{lem:sp-mixed} and~\ref{lem:sp-incidence} yield
\begin{align}\label{eq:sp-keychain}
\Phi_{m,n}(C,\eta)=\omega(\eta)
&\leq\|\omega\|_*
=\|\mathsf A\mathsf S\mathsf B^T\|_{\op}\notag\\
&\leq\sqrt{\a_m}\,2^{(m+n-2)/2}\sqrt{\a_n}
=\Lambda_{m,n}.
\end{align}

To prove that the constant is optimal, we consider the product face
$K_{01}\times K'_{01}$. Since
$\d_{01}=e_1\in G_m$ and
$\d'_{01}=e'_1\in G_n$, set
\be
C=e_1(e'_1)^T\in M_{m,n}(\FF).
\ee
In $\R^m$, let
\be
\nu_{01}
=\frac{w_0-w_1}{\sqrt{2(m+1)/m}},
\qquad
\tau_m=q_m\sqrt{\frac{m+1}{2m}}
=\sqrt{\frac{\a_m}{2^{m-1}}}.
\ee
Then $\nu_{01}$ is a unit normal to $K_{01}$, and
\be
a_g\cdot\nu_{01}=\frac{q_m}{\sqrt{2(m+1)/m}}\sum_{r=1}^m g_rw_r\cdot(w_0-w_1).
\ee
Note that by \eqref{eq:sp-wframe}, when $r\ne 1$, $w_r\cdot(w_0-w_1)=(-\frac 1m)-(-\frac 1m)=0$, hence 
\be\label{K01calibration} a_g\cdot\nu_{01}=\frac{q_m}{\sqrt{2(m+1)/m}} g_1w_1\cdot(w_0-w_1)=\frac{q_m}{\sqrt{2(m+1)/m}} g_1(-\frac{m+1}{m})=-\tau_m g_1.\ee
Similarly for $\R^n$,
\be\label{K01calibration2}
a'_h\cdot\nu'_{01}=-\tau_n h_1.
\ee
Now choose $\eta=(\nu_{01},0)\wedge(0,\nu'_{01})$. Note that by definition of $C$, 
\be
\mathsf B_C(g,h)=1
\quad\Longleftrightarrow\quad
g_1=h_1=1.
\ee
There are $2^{m-1}2^{n-1}$ such pairs, and each contributes
$\tau_m\tau_n$ after \eqref{K01calibration} and \eqref{K01calibration2}. Hence
\[
\Phi_{m,n}(C,\eta)
=2^{m+n-2}\tau_m\tau_n
=2^{(m+n-2)/2}\sqrt{\a_m\a_n}
=\Lambda_{m,n}.
\]
Thus the constant is optimal.

By symmetry of the regular simplices, the same equality
holds at every product face $K_{ij}\times K'_{ab}$.
\qed

\subsection{Minimality in the spanning competitor class}\label{subsec:simplex-boundary}

We now return from the pointwise estimate to the global
comparison. As before, the only information that will be kept
from the competitor is the family of prescribed boundary
currents. Hence, as before, we first prove the minimality of $E_{m,n}^0$ in the following larger class.
\begin{defn}[The spanning competitor class]\label{def:sp-currentclass}
Let $\mathscr C_{m,n}(\mathcal D_{m,n})$ be the class of all compact sets $F\subset\mathcal D_{m,n}$ such that  $F$ is $d$-rectifiable, $\H^d(F)<\infty$, and $F$ supports $mn$ currents $\Gamma_{rs}\in\I_d(X\times X';\FF), 1\le r\le m,1\le s\le n$ with
\be\label{eq:sp-prescribed-boundary}
\spt\Gamma_{rs}\subset F,\qquad
\partial\Gamma_{rs}=\partial(D_{e_r}\times D'_{e_s})
\quad(1\leq r\leq m,\ 1\leq s\leq n).
\ee
\end{defn}
This is the same type of spanning competitor class as in Section~\ref{sec:minimality}. 

\begin{lem}\label{lem:sp-fixed-projection}
Take any $F\in\mathscr C_{m,n}(\mathcal D_{m,n})$, and choose the $mn$ currents $\Gamma_{rs}$ as in its definition. Define, for nonzero $g,h$, 
\be\label{eq:sp-Gamma}
\Gamma_{g,h}=\sum_{r=1}^m\sum_{s=1}^n g_rh_s\Gamma_{rs}.
\ee
Then for every $g\in G_m^\times, h\in G_n^\times$, the projection current
\be\label{eq:sp-degreeone}
(\Pi_{g,h})_{\s}\Gamma_{g,h}=\chain{Q_g\times Q'_h}.
\ee
Moreover,
\be\label{eq:sp-projectionlower}
A_gA'_h\leq\int_{\car\Gamma_{g,h}}
J_d(\Pi_{g,h}|_{T_xF})\,d\H^d(x).
\ee
\end{lem}
\nd
By \eqref{eq:sp-prescribed-boundary} and \eqref{eq:sp-product-linear}, we have
\be
\partial\Gamma_{g,h}=\partial(D_g\times D'_h).
\ee
Meanwhile Lemma~\ref{lem:sp-projection} gives
\be
(\Pi_{g,h})_{\s}(D_g\times D'_h)
=((\rho_g)_{\s}D_g)\times((\rho'_h)_{\s}D'_h)
=\chain{Q_g\times Q'_h}.
\ee
Thus
\be
Z=(\Pi_{g,h})_{\s}\Gamma_{g,h}+\chain{Q_g\times Q'_h}
\ee
is a compactly supported integral $d$-current modulo two in $a_g^\perp\times {a'_h}^\perp$, and
\be
\partial Z
=(\Pi_{g,h})_{\s}\bigl(\partial\Gamma_{g,h}
+\partial(D_g\times D'_h)\bigr)=0.
\ee
The target plane $a_g^\perp\times {a'_h}^\perp$ also has dimension $d$, so the constancy Lemma~\ref{lem:topcycle} gives $Z=0$. This proves \eqref{eq:sp-degreeone}.

Now we apply Lemma~\ref{lem:pushmass} to \eqref{eq:sp-degreeone}, and get
\be\begin{split} \int_{\car\Gamma_{g,h}}J_d(\Pi_{g,h}|_{T_xF})\,d\H^d(x)&\ge \mass((\Pi_{g,h})_{\s}\Gamma_{g,h})=\H^d(Q_g\times Q'_h)\\
&=\H^{m-1}(Q_g)\H^{n-1}(Q'_h)=A_gA'_h.
\end{split}\ee
This gives \eqref{eq:sp-projectionlower}.
\qed

The remaining step is to sum the projection inequalities with the correct weights.

\begin{thm}[Minimality in the spanning competitor class]\label{thm:sp-boundary-minimality}
For every $F\in\mathscr C_{m,n}(\mathcal D_{m,n})$,
\be\label{eq:sp-areabound}
\H^d(F)\geq M_mM_n
=\H^d(E^0_{m,n}).
\ee
\end{thm}
\nd Fix any $F\in\mathscr C_{m,n}(\mathcal D_{m,n})$, choose the $mn$ currents $\Gamma_{rs}$ as in the definition of $\mathscr C_{m,n}(\mathcal D_{m,n})$. Let $C(x)$ be the multiplicity matrix for $\Gamma_{rs}, 1\le r\le m, 1\le s\le n$ on the common carrier $F$, after Lemma~\ref{lem:common}. Then outside one common null set,
\be\label{eq:sp-activecarrier}
x\in\car\Gamma_{g,h}\quad\Longleftrightarrow\quad g^TC(x)h=1,
\ee
where $\Gamma_{g,h}$ is as defined in \eqref{eq:sp-Gamma}. Let $\eta_x$ be a measurable unit simple 2-vector field on $F$, so that $\eta_x\perp T_xF$ almost everywhere.

Multiply \eqref{eq:sp-projectionlower} by $A_gA'_h>0$ and sum over all nonzero $g,h$. Then by equations \eqref{eq:sp-activecarrier}, \eqref{eq:sp-jacobian}, and Theorem \ref{thm:sp-pointwise}, we have
\be\label{eq:sp-comparisonchain}\begin{split}
&\left(\sum_{g\in G_m^\times}A_g^2\right)
\left(\sum_{h\in G_n^\times}{A'_h}^2\right)=\sum_{g\in G_m^\times,h\in G_n^\times}A_g^2{A'_h}^2\\
&\leq\sum_{g,h}A_gA'_h\int_F\mathbf1_{\{g^TC(x)h=1\}}
 J_d(\Pi_{g,h}|_{T_xF})\,d\H^d(x)\\
&=\int_F\Phi_{m,n}(C(x),\eta_x)\,d\H^d(x)\leq\Lambda_{m,n}\H^d(F).
\end{split}\ee
By \eqref{eq:sp-frame}, the left-hand side is $mn\a_m\a_n$. Hence by Lemma \ref{lem:sp-model-volume},
\be\label{eq:sp-numerical-lower}
\H^d(F)\geq
\frac{mn\a_m\a_n}{\Lambda_{m,n}}
=\frac{mn\sqrt{\a_m\a_n}}{2^{(m+n-2)/2}}
=M_mM_n.
\ee

It remains to verify that the right-hand side is the measure of the reference set $E^0_{m,n}$. The set $E^0_{m,n}$ is the union of all $K_{ij}\times K'_{ab}$, with disjoint relative interiors; the intersection of two distinct product faces has dimension at most $d-1$. Thus 
\begin{align*}
\H^d(E^0_{m,n})&=\sum_{i<j,a<b}\H^{m+n-2}(K_{ij}\times K'_{ab})
=\sum_{i<j,\ a<b}\H^{m-1}(K_{ij})\H^{n-1}(K'_{ab})\\
&=\left(\sum_{i<j}\H^{m-1}(K_{ij})\right)
  \left(\sum_{a<b}\H^{n-1}(K'_{ab})\right)
=M_mM_n.
\end{align*}
The currents $D_{e_r}\times D'_{e_s}$ show that
$E^0_{m,n}\in\mathscr C_{m,n}(\mathcal D_{m,n})$, so the lower bound is attained.
\qed

\subsection{Almgren minimality of $E_{m,n}=S_m\times S_n$}\label{subsec:simplex-almgren}

We now come to the proof of the Almgren minimality of $E_{m,n}=S_m\times S_n$. It remains only
to check that every Almgren deformation produces a competitor in the above spanning class.

\noindent \textbf{Proof of Theorems~\ref{thm:simplex-main} and~\ref{main:simplex}. }

Take all the notations as before. Since $E_{m,n}=S_m\times S_n$ is a cone, it is enough to prove its Almgren minimality in $\mathring{\mathcal D_{m,n}}$.
Let $\{\varphi_t\}_{0\leq t\leq1}$ be any Almgren deformation in $\mathring{\mathcal D_{m,n}}$, write $f=\varphi_1$.
Set
\be\label{eq:sp-deformedset}
F:=f(E_{m,n})\cap\mathcal D_{m,n}
=f(E_{m,n}\cap\mathcal D_{m,n})=f(E^0_{m,n}).
\ee
This is a finite union of Lipschitz images of compact polyhedra, hence is compact and $d$-rectifiable, with $\H^d(F)<\infty$.

Define
\be\label{eq:sp-deformed-currents}
\Gamma_{rs}=f_{\s}(D_{e_r}\times D'_{e_s}),
1\leq r\leq m,\ 1\leq s\leq n.
\ee
These $\FF$ integral currents are supported in $F$. Lemma~\ref{lem:sp-current} and the product boundary formula give
\begin{align}\label{eq:sp-product-boundary}
\partial(D_g\times D'_h)
&=(\partial D_g)\times D'_h+D_g\times\partial D'_h,\notag\\
\spt\partial(D_g\times D'_h)
&\subset(\partial\Delta^m\times\Delta^n)
\cup(\Delta^m\times\partial\Delta^n)
=\partial\mathcal D_{m,n}.
\end{align}
The map $f$ is the identity near these boundaries, so
\be
\partial\Gamma_{rs}
=f_{\s}\partial(D_{e_r}\times D'_{e_s})
=\partial(D_{e_r}\times D'_{e_s}).
\ee
Thus all conditions in Definition~\ref{def:sp-currentclass} are verified, and therefore
$F\in\mathscr C_{m,n}(\mathcal D_{m,n})$. 

We apply Theorem~\ref{thm:sp-boundary-minimality} and \eqref{eq:sp-deformedset}, and get
\be
\H^d(f(E_{m,n})\cap\mathcal D_{m,n})
\geq M_mM_n
=\H^d(E_{m,n}\cap\mathcal D_{m,n}).
\ee
Subtract the common term
$\H^d(E_{m,n}\cap f(E_{m,n})\cap\mathcal D_{m,n})$ from both sides, and we get
\be\label{eq:sp-final-Almgren}
\H^d(E_{m,n}\bs f(E_{m,n}))
\leq\H^d(f(E_{m,n})\bs E_{m,n}).
\ee
This gives the Almgren minimality of $E_{m,n}=S_m\times S_n$.\qed

\subsection{The difficulty with the product of three or more factors}\label{subsec:simplex-scope}

It is natural to ask whether the same argument works for three or more factors. The current-theoretic part causes no immediate difficulty: the $\FF$ cycle families, the multilinear $\FF$ multiplicities, and the fixed-boundary projection identities can all be constructed in the same way. The main obstruction appears in the pointwise estimate Theorem~\ref{thm:sp-pointwise}, which is the core estimate that allows arbitrary folding and overlap.

\bigskip

\textbf{For example}, for the product of three orthogonal copies of 1-dimensional $\Y$ set, $\Y\times \Y\times \Y$, let $X_1,X_2,X_3$ be three mutually orthogonal copies of $\R^2$, and for $j=1,2,3$, choose unit vectors $u_j,v_j$ with $u_j\cdot v_j=-1/2$, and set $a_g^{(j)}=\sqrt3(g_1u_j+g_2v_j)$ (with $q_2=\sqrt 3$) for $g\in\FF^2\bs\{0\}$ and $\eta=u_1\wedge u_2\wedge u_3$. Define the trilinear form, with addition in $\FF$,
\be
 C(g,h,\ell)=g_1h_1\ell_1+g_1h_2\ell_2+g_2h_1\ell_2
             +g_2h_2\ell_1+g_2h_2\ell_2.
\ee
For each factor $j=1,2,3$, in the order $g=10,01,11$, we have
\be
(|a_g^{(j)}\cdot u_j|)_{10,01,11}
=\sqrt3\left(1,\frac12,\frac12\right).
\ee
Note that 
\be\left|
\left\langle
\eta,
a_g^{(1)}\wedge a_h^{(2)}\wedge a_\ell^{(3)}
\right\rangle
\right|=|a_g^{(1)}\cdot u_1||a_h^{(2)}\cdot u_2||a_\ell^{(3)}\cdot u_3|
=
3\sqrt3\,\rho(g)\rho(h)\rho(\ell),\ee
where $\rho(10)=1,
\rho(01)=\rho(11)=\frac12$.
For fixed $g=10,01,11$, summing over all $(h,\ell)$ such that
$C(g,h,\ell)=1$ gives respectively $\frac{11}{4},\frac54,\frac{11}{8}.
$
 Hence
\be\label{eq:sp-three-factor-obstruction}
\begin{split}\sum_{\substack{g,h,\ell\in\FF^2\bs\{0\}\\C(g,h,\ell)=1}}
\left|\pair{\eta}{a_g^{(1)}\wedge a_h^{(2)}\wedge a_\ell^{(3)}}\right|
=\sum_{g\in \FF^2\bs\{0\}}\sum_{\substack{h,\ell\in\FF^2\bs\{0\}\\C(g,h,\ell)=1}}
\left|\pair{\eta}{a_g^{(1)}\wedge a_h^{(2)}\wedge a_\ell^{(3)}}\right|
\\
=3\sqrt 3(\frac{11}{4}+\frac{5}{4}+\frac{11}{8})=\frac{129\sqrt3}{8}>27
=\bigl(\sqrt2\sqrt{\a_2}\bigr)^3.
\end{split}\ee
Here $\a_2=9/2$, and the bound $27$ on the right is exactly what the weighted projection sum would need to recover the reference measure $\H^3(\Y^3\cap(\Delta^2)^3)=27$. 

Here in the example, $\eta=u_1\wedge u_2\wedge u_3$ is already a product normal 3-vector. Moreover, for the multiplicity trilinear form $C$, after fixing any two of the three indices, exactly two of the three possible values of the remaining index satisfy $C(g,h,\ell)=1$. Thus every coordinate fiber of the incidence tensor has exactly two active entries. Nevertheless, the weighted trilinear sum exceeds the required bound $27$.

One final remark is, \emph{the calculation above disproves only the direct extension of the present calibration estimate, not Almgren minimality of the multi-factor product.} 

\renewcommand\refname{References}
\begingroup
\small
\linespread{0.95}\selectfont
\bibliographystyle{plain}
\bibliography{reference}
\endgroup
\end{document}